\documentclass[a4paper]{article} 

\usepackage{hhline}
\usepackage{caption}
\usepackage{booktabs}
\usepackage[normalem]{ulem}
\usepackage{amsfonts}
\usepackage{amsmath, mathtools} 
\usepackage{amssymb}
\usepackage{graphicx,wrapfig}
\usepackage{amsthm}
\usepackage{dsfont}    
\usepackage{microtype} 

\usepackage{hyperref}
\usepackage{mathrsfs}
\usepackage{caption}
\usepackage{subcaption}
\usepackage{color}
\usepackage{url}

\usepackage[utf8]{inputenc}
\usepackage[english]{babel}

\newtheorem{theorem}{Theorem}
\theoremstyle{definition}
\newtheorem{definition}{Definition}
\newtheorem*{algorithm*}{Algorithm}
\newtheorem{lemma}[theorem]{Lemma}

\newtheorem{corollary}[theorem]{Corollary}

\theoremstyle{remark}
\newtheorem*{remark}{Remark}
\newtheorem*{acknowledgement}{Acknowledgement}

\newcommand{\N}{\mathbb{N}}

\newcommand{\C}{\mathbb{C}}
\newcommand{\Z}{\mathbb{Z}}
\newcommand{\R}{\mathbb{R}}

\newcommand{\md}{\mathbb{M}}

\newcommand{\s}{\mathbb{S}^2}
\newcommand{\sdd}{\mathbb{S}^d}

\newcommand{\dx}{\ \mathrm{d}x}

\newcommand{\dt}{ \mathrm{d}t}
\newcommand{\dm}{\ \mathrm{d}\mu}

\newcommand{\id}{\mathrm{id}}

\newcommand{\onehalf}{\frac{1}{2}}

\newcommand{\fatone}{\mathds{1}}
\newcommand{\scd}{\mathcal{D}_{\infty}}
\newcommand{\lpd}{\mathcal{D}_{p}}

\newcommand{\ra}{\rightarrow}

  \date{}
\begin{document}
	\title{Uniform distribution via lattices: from point sets to sequences}
	\author{Damir Ferizovi\'{c}\thanks{KU Leuven, Department of Mathematics,  3001 Leuven (BE). damir.ferizovic@kuleuven.be
			   The
			author thankfully acknowledges support by the Methusalem grant METH/21/03 – long term structural funding of the Flemish Government; and was partially funded for a short stay in Santander, Spain by
			Grant PID2020-113887GB-I00 of MCIN/ AEI /10.13039/501100011033. }}
\maketitle 
\begin{abstract}
 In this work we construct many sequences $S=S^\Box_{b,d}$ and  $S=S^\boxplus_{b,d}$ in the $d$--dimensional unit hypercube, which for $d=1$ are (generalized) van der Corput sequences or Niederreiter's $(0,1)$-sequences in base $b$, respectively. Further, we introduce the notion of $f$-subadditivity and use it to define discrepancy functions which subsume the notion of $L^p$-discrepancy, Wasserstein $p$-distance,  and many more methods to compare empirical measures to an underlying base measure.  
 We will relate bounds for a given discrepancy function $\mathscr{D}$ of the multiset of projected lattice sets (treated as empirical measures), $P(b^{-m}\Z^d$), to  bounds of $\mathscr{D}(E_{Z_N})$, i.e. the initial segments of the sequence $Z=P(S)$ for any $N\in\N$. We show that this relation holds in any dimension $d$ and for any map $P$ defined on a hypercube, for which bounds on $\mathscr{D}(E_{P(b^{-m}\Z^d+v)})$ can be obtained. 
	We apply this theorem in $d=1$ to obtain bounds for the $L^p$--discrepancy of van der Corput and Niederreiter (0,1) sequences  in terms of digit sums for all $0<p\leq \infty$. In $d=2$ an application of our construction yields many sequences on the two-sphere, such that the initial segments $Z_N$ have small $L^\infty$--discrepancy.
\end{abstract}

\section{Outline}

In this work we consider sequences $Z=(z_1,z_2,\ldots)$ in an arbitrary non-trivial set $\md$, where $Z$ will be the image of a certain sequence $S=(s_0,s_1,\ldots)$ in $[0,1)^d$ under some map $P:[0,1)^d\ra\md$ (note the index shift).  We write $Z=P(S)$ if $z_{n+1}=P(s_n)$ and denote   \textbf{initial segments} by $Z_N=(z_1,\ldots,z_{N})$.  For any  $M\subset [0,1)^d$, $P(M)$ shall be understood as a multiset, i.e. repeated elements are allowed, and we will often treat $Z_N$ as the multiset $\{z_1,\ldots,z_{N}\}$. Empirical measures with atoms at elements from a multiset $W$ will be denoted by $E_W$. We will present a simple method to construct sequences $S$ in $[0,1)^d$ with built in self-similarity, which for $d=1$ are (generalized) van der Corput sequences. Our main result,  Theorem \ref{thm_main}, allows us to transfer bounds: if  for a given discrepancy function $\mathscr{D}$ and all lattice sets $\ell^d_{b^{m}}$ one has  bounds on $\mathscr{D}(E_{P(\ell^d_{b^{m}}+v)})$, we will derive upper bounds for $\mathscr{D}(E_{Z_N})$ -- the notation is explained in Section \ref{sec_IntroResults}, Definitions \ref{def_perturbedLatticeSet} and \ref{def_discrepancyFunction}.
 Applications of our main result are found in Section \ref{sec_UniforDistrThe}, where we prove many bounds  for the  van der Corput sequences and their generalizations, and then give many examples of sequences on the two-dimensional unit sphere $\s$ with small $L^\infty$--discrepancy. This section also contains  Theorem \ref{thm_SizeOfSetsB}, which "counts" the number of integers $K\in\{1,\ldots,N\}$  for certain sequences $S$ that satisfy much tighter upper bounds for $\mathscr{D}(E_{S_K})$. The construction of the fundamental sequence $S^\Box_{b,d}$ is found in Section \ref{sec_SboxSequence}, and the proof of our main result for these sequences follows in Section \ref{subsec_proofThm1LmLatticeSets}. \textit{Indeed, if the reader  wants a quick understanding of  Theorem \ref{thm_main}, only  Section \ref{sec_IntroResults} and \ref{sec_Theorem1SBox} need to be read (see Figures \ref{fig:Lem8Visualization} and \ref{fig:Thm1Visualization}).} The  sequences $S^\boxplus_{b,d}$ are constructed in Section \ref{subsec_GuidedSequences} and for $d=1$ coincide with Niederreiter's (0,1)--sequences. The proof of our main result for these sequences  then follows in Section \ref{subsec_proofThm1LmPerturbed}.  Section \ref{sec_proofSizeofSetsB} contains the proof of Theorem \ref{thm_SizeOfSetsB}. Section \ref{sec_proofApplications} collects the proofs of the  applications in uniform distribution theory, where we highlight the fact that our construction and Theorem \ref{thm_main} simplify matters to such an extent that half the results  follow from elementary computations alone. In Section 8 we extend our construction of $S^\boxplus$ to  special self--similar sets.

\section{Basic Definitions and Main Result}\label{sec_IntroResults}

Our main result needs information on the discrepancy of perturbed lattice sets.
\begin{definition}\label{def_perturbedLatticeSet}
	 A \textbf{lattice set of resolution $K$} shall be a set of the form $\ell^d_K:=\big(\frac{1}{K} \Z^d\big)\cap[0,1)^d$, where $K\in\N$. Let $\omega_N\subset [0,1)^d$ with\footnote{For a finite multiset $S$, we write $|S|$ for the number of elements in it (counting multiplicities).}  $|\omega_{N}|=N=K^d$.  We call $\omega_{N}$ a \textbf{perturbed lattice set of resolution $K$}  if 
	for any $q\in \ell^d_K$ we have
	\[
	\Big|\omega_N\cap \big(q+\tfrac{1}{K}[0,1)^d\big)\Big|=1.
	\]
	Let  $T_v(\omega_{N})=\omega_{N}+v$ and call $T_v$ a \textbf{valid translate} for $\omega_{N}$ if $T_v(\omega_{N})\subset[0,1)^d$.
\end{definition}

A perturbed lattice set of resolution $K$ is thus the result of allowing each lattice point to be placed arbitrarily in its lattice square (or hypercube) of side length $1/K$. See Figure \ref{fig:Sboxplus} a) for an example with $K=d=2$ and $\omega_{4}=\{q_0,q_1,q_2,q_3\}$. 

\begin{definition}\label{def_fSubLinear}
	Let $A_C$ be an additively-closed subset of an additive group, i.e. $\forall x,y\in A_C$ we have $x+y\in A_C$.  We shall say that a function $\mathcal{L}$ is $f$-\textbf{subadditive} over $A_C$ if $\mathcal{L}:A_C\ra\R$ and  
	$f:\N\ra\R$ is a function satisfying 
	\[
	\mathcal{L}\Big(\sum_{j=1}^{n}x_j\Big)\ \leq\ f(n)\cdot \sum_{j=1}^{n} \mathcal{L}(x_j), \hspace{1cm}\forall n\in\N,\ \forall x_1,\ldots,x_n\in A_C.
	\]
\end{definition}
An important example over $\R$ is given by  $\mathcal{L}(x):=|x|^p$; for $p\geq 1$ it  is $f$-subadditive with $f(n)=n^{p-1}$, by the triangle and 
Jensen inequalities:
\[
\Big|\sum_{j=1}^nx_j\Big|^p\leq\Big(\sum_{j=1}^n|x_j|\Big)^p\leq n^{p-1}\sum_{j=1}^n|x_j|^p.
\] 
If   $\mathcal{L}(x):=|x|^p$ with $0<p<1$, then  it is $f$-subadditive over $\R$ with $f(n)\equiv 1$, which follows from the triangle inequality, continuity and the fact that for $(x_1,\ldots,x_m)=\vec{x}\in\big(\R_{>0}\big)^m$,
\begin{equation}\label{eq_subadditivePless1}
	g_p(\vec{x}):=\Big(\sum_{j=1}^n x_j\Big)^{p}- \sum_{j=1}^nx_j^p\ <\ 0.
\end{equation}
Equation \eqref{eq_subadditivePless1}  follows from Euler's homogeneous function theorem, since $g_p(r\vec{x})=r^pg_p(\vec{x})$ for $r> 0$, thus
\[
p\cdot g_p(\vec{x})=\sum_{j=1}^{n}x_j\frac{\partial g_p}{\partial x_j}(\vec{x})\ < \ 0,\hspace{0.3cm}\mbox{since }x_j>0\mbox{ and } \frac{\partial g_p}{\partial x_j}(\vec{x})<0.
\]
\begin{definition}\label{def_discrepancyFunction}
Consider the set $\mathcal{E}_\md$ of \textbf{empirical measures} $E_{Z_N}:= \sum_{j=1}^{N}\delta(z_j)$ on some set $\md$, $z_j\in Z_N$, where $\delta(x)$ is the delta-measure at the point $x$: 	
	\[
	\mathcal{E}_\md=\Big\{E_{Z_N}\ \Big|\ Z_{N}=\{z_1,\ldots,z_{N}\},\ Z_N\mbox{ a multiset},\ N\in\N,\ z_j\in\md\Big\}.
	\]
	A function-pair $(\mathscr{D},f)$ shall be called \textbf{admissible} if $\mathscr{D}$ is $f$-subadditive over $\mathcal{E}_\md$, i.e. $\mathscr{D}:\mathcal{E}_\md\ra\R$ and $\mathscr{D}(\sum^n\nu_j)\leq f(n)\sum^n \mathscr{D}(\nu_j)$. We call $\mathscr{D}$ a \textbf{discrepancy function} if there is an $f$ that makes  $(\mathscr{D},f)$  admissible. 
\end{definition}
Even though $\mathscr{D}$ is defined on empirical measures, we advise the reader to think of $\mathscr{D}$ as a function on point sets of $\md$, and in our main results the element $\sum^n\nu_j$ from the definition above will be the initial segment of a sequence $Z$, and each $\nu_j$ will be the projection of a perturbed lattice set. We introduced $f$-subadditive and discrepancy functions in order to state our main result  in full generality. As we shall see in the remarks after Definition \ref{def_LpDiscrepancy}, our admissible function-pair subsumes many notions of discrepancy.

\vspace{0.2cm}
Our main result Theorem \ref{thm_main} is heavy in notation and in order to  give the gist of it quickly we first state a simple version of it.
	
\vspace{0.2cm}
\textbf{Simplified Main Result.} Let $P:[0,1)^d\ra\md$ be a map.
	Fix an admissible function-pair  $(\mathscr{D},f)$ and integer $b>1$. Fix any sequence $S^\Box_{b,d}=(s_0,s_1,\ldots)$ as constructed in Section \ref{sec_SboxSequence} and write
	\[
	Z=(z_1,z_2,\ldots)=:P(S^\Box_{b,d}) \hspace{0.5cm}\mbox{where} \hspace{0.3cm} z_{j+1}=P(s_j)
	\]
	   for the mapped sequence, then we have for every $N\in\N$ the bound
	\[
	\mathscr{D}\big(E_{Z_{N}}\big)=\mathscr{D}\big(E_{\{z_1,\ldots,z_{N}\}}\big)\ \leq\ f\Big(\sum_{j=0}^n\epsilon_j\Big)\cdot \sum_{j=0}^n\epsilon_j h_{j}
	\]
	where the $n,\ \epsilon_j$'s appear in the $b^d$--adic expansion of $N=\sum_{j=0}^{n}\epsilon_j b^{dj}$ and 
	\[
	h_m:=\sup_{v\in[0,b^{-m})^d}\mathscr{D}\big(E_{P((b^{-m} \Z^d+v)\cap[0,1)^d)}\big),
	\]
	i.e. the $h_j$'s are the supremum of the discrepancy function $\mathscr{D}$ over all point sets that arise as images of shifted lattice sets under $P$.

	\begin{remark}
			It is in general very hard to bound the discrepancy of every initial segment of a sequence because for instance such point sets do not satisfy optimality conditions which could be used in a proof. 
		Our theorem thus provides a powerful and simple tool to give such bounds via discrepancies of point sets with a well understood structure, and in order to apply this theorem all we really need are bounds on these $h_m$. 
	\end{remark}
	\vspace{0.2cm}

	We will apply this theorem for $d=1$ with $\mathscr{D}$ being the $L^p$-discrepancy on $[0,1]$ and where $S^\Box_{b,1}$ turn out to  be generalized van der Corput sequences. The numbers $h_m$ can be bounded by the same constant and  we obtain novel bounds for discrepancies of these sequences for any $p\in (0,\infty]$ in terms of sums of digits of $N$, i.e. terms of the form $\sum_{j=0}^n\epsilon_j$, see Theorem \ref{thm_ApplvdC}.
	
	\vspace{0.2cm}
	The "full" main result is much heavier on notation and the reason for it is to implicitly address the following three issues.
	 
	\begin{enumerate}
		\item Real world applications deal with very large ordered sets instead of sequences and our theorem  holds with simpler requirements in this finite case, where in the definition of $h_m$ instead of a supremum only the maximum over a finite set needs to be computed, i.e. for a given $N=\sum_{j=0}^{n}\epsilon_j b^{dj}$ all we need is the maximum over shifts $v$ in $W^{n,b}_m= [0,b^{-m})^d\cap \big(\Z^d\cdot b^{-n-1}\big)$.
		
		\item The theorem also holds for generalized constructions like $S^\boxplus_{b,d}$ as in Section \ref{subsec_GuidedSequences}, but this complicates the description of the domain $\Omega^n_m$  over which we have to take the supremum in the definition of $h_m^n$.
		
		\item If $N=\sum_{j=0}^{n}(b^{d}-1) b^{dj}$  then  $N+1=b^{dn+d}$ and the upper bounds in the theorem experience a big drop from $\mathscr{D}\big(E_{Z_{N}}\big)$ to $\mathscr{D}\big(E_{Z_{N+1}}\big)$. There is a version of our result that behaves more continuously in this regard if $\mathscr{D}$ is additionally  a discrepancy function over $\mathcal{E}_\md-\mathcal{E}_\md$, i.e.
		\[
		\mathscr{D}:\big(\mathcal{E}_\md-\mathcal{E}_\md\big)\ra\R
		\]  and with the same $f$ 
		\[
		\mathscr{D}\Big(\sum^n\nu_j-\sum^k\mu_j\Big)\leq f(n+k)\sum^n \mathscr{D}(\nu_j)
		+f(n+k)\sum^k \mathscr{D}(\mu_j).
		\]
	\end{enumerate}
	 
	Theorem \ref{thm_main} holds for all sequences $S^\boxplus_{b,d}$ but for the simpler constructions $S^\Box_{b,d}$ we can relax our conditions. Let $\mathscr{Z}^u$ represent  the initial segment of length $b^{ud}\leq \infty$ of either $P(S^\boxplus_{b,d})$ or  $P(S^\Box_{b,d})$ where the choice is up to the practitioner. The option between $S^\boxplus_{b,d}$ and  $S^\Box_{b,d}$ influences the set $L_m$ and thus the description of the domain $\Omega^n_m$ which  determines the difficulty in computing bounds for $h^n_m$.

\begin{theorem}[Main result]\label{thm_main}Let $P:[0,1)^d\ra\md$ be a map.
	 Fix an admissible function-pair  $(\mathscr{D},f)$, integer $b>1$ and $u\in\N\cup\{\infty\}$. Let $(L_m,\ \mathscr{Z}^u)$ have one of two meanings for   $0\leq m\leq u$ (where $\mathscr{Z}^\infty$ is the set of sequences):
	 \begin{enumerate}
	 	\item   $L_m=\{\ell^d_{b^m}\}$ and $\mathscr{Z}^u=\big\{ P(S^\Box_{b,d})_{b^{ud}}\ |\ S^\Box_{b,d}\mbox{ as in Section \ref{sec_SboxSequence}} \big\}\subset\md^{b^{ud}};$
	 	
	 	\item $L_m=\big\{\omega\subset[0,1)^d\ |\ \omega\mbox{ is a perturbed lattice set of resolution  }b^m  \big\}$, and 
	 	\[
	 	\mathscr{Z}^u=\big\{ P(S^\boxplus_{b,d})_{b^{ud}}\ |\ S^\boxplus_{b,d}\mbox{ as in Section \ref{subsec_GuidedSequences}} \big\}\subset\md^{b^{ud}}.
	 	\] 
	 \end{enumerate}
	 Define the sets $W^{k,b}_m= [0,b^{-m})^d\cap \big(\Z^d\cdot b^{-k-1}\big)$ and let $W^{\infty,b}_m=\bigcup_{k=0}^\infty W^{k,b}_m$. Set $\Omega_m^\star:=\big\{(\omega,v)\in L_m\times W^{\star,b}_m\ |\ T_v\mbox{ is a valid translate for }\omega  \big\}$ and let, for $k\in\N$,
	\[
	\sup_{(\omega,v)\in \Omega_m^k}\mathscr{D}\big(E_{P(T_{v}(\omega))}\big)=: h_m^k\ \leq\ h_m^\infty:=\sup_{(\omega,v)\in \Omega_m^\infty}\mathscr{D}\big(E_{P(T_{v}(\omega))}\big).
	\]
	Let $N< b^{ud}$ with $b^d$--adic expansion $N=\sum_{j=0}^{n}\epsilon_j b^{dj}$. Then  for any  $Z\in \mathscr{Z}^u$,
	 \begin{equation}\label{eq_mainResultN}
	 	\mathscr{D}\big(E_{Z_{N}}\big)\ \leq\ f\Big(\sum_{j=0}^n\epsilon_j\Big)\cdot \sum_{j=0}^n\epsilon_j h_{j}^{n}\ \leq\ f\Big(\sum_{j=0}^n\epsilon_j\Big)\cdot \sum_{j=0}^n\epsilon_j h_{j}^{\infty}.
	 \end{equation}
Let $N'=1+\sum_{j=0}^{n}\big(b^d-1-\epsilon_j\big)\cdot b^{dj}$ such that   $N=b^{(n+1)d}-N'$. If $\mathscr{D}$ is additionally  $f$-subadditive over $\big\{\nu_1-\nu_2\ |\ \nu_1,\nu_2\in\mathcal{E}_\md \big\}$, then with $\mathscr{Z}^u,\{h_m^n\},$ and $N'$ as above, we obtain 
	 \begin{equation}\label{eq_mainResultNprime}
	\mathscr{D}\big(E_{Z_{N}}\big)\leq
	f\Big(2+\sum_{j=0}^n\big(b^{d}-1-\epsilon_j\big)\Big)\cdot\Big(h_{n+1}^{n}+h_{0}^{n}+\sum_{j=0}^n(b^d-1-\epsilon_j\big)h_j^{n}\Big).
\end{equation}
\end{theorem}

\begin{remark}
	Clearly $n=\lfloor\frac{\log(N)}{d\log(b)}\rfloor$ and $0\leq \epsilon_j<b^d$. $\md$ and $P$ can be arbitrary; the problem is to find upper  bounds on $h_m^n$ or $h^\infty_m$. In applications we expect $\md$ to be  a  metric measure space and $P$  to be measure preserving. Examples of such objects are found in \cite{BeltranFerizovicLopez, GriepentrogEtAl, Holhos}. 
\end{remark}

\section{Applications in Uniform Distribution Theory}\label{sec_UniforDistrThe}

Uniform distribution  refers to a special discretization of   a compact set by a collection of finite points. Some of the most common ways to quantify this property is in terms of the $L^p$--discrepancy, which we introduce below.  The start of the investigation in such matters on the unit hypercube is attributed to Weyl \cite{Weyl} (see \cite{HlawkaBinder} for a historic review). 
\begin{definition}\label{def_LpDiscrepancy}
	Let  $\chi_A$ be the indicator function of  $A\subset \md$, let $(\md,d_\md,\mu)$ be a metric measure space with $r_\md$ the diameter of $\md$,  $r_\md\leq \infty$,  and let $\mu$ be  a Borel probability measure with Borel sigma algebra $\Sigma(\md)$. A ball with center $w\in\md$ and radius $t\in(0,r_\md)$ is the set
	$$B(w,t)=\big\{x\in\md\ |\ d_\md(w,x)\leq t  \big\}.$$
	The  \textbf{$L^p$-discrepancy}  $\sqrt[p]{\lpd(Z_N)}$ of a multiset $Z_N$ of $\md$, $p\in(0,\infty)$, is defined via
	\begin{equation*}
	\lpd(Z_N)=\frac{1}{N^p}\int_{0}^{r_\md}\int_{\md} |D_{B(w,t)}(Z_N)|^p\dm(w)\dt,
	\end{equation*}
	where, with $Z_N=\{z_1,\ldots,z_N\}$ and empirical measure $E_{Z_N}$, we use the abbreviation
	\[
	D_{\star}(Z_N)= \sum_{j=1}^{N}\chi_{\star}(z_j)-N\mu(\star)=E_{Z_{N}}(\star)-E_{Z_{N}}(\md)\mu(\star).
	\]
	Furthermore, we define for a collection of $\mu$--measurable test-sets $T\subset \Sigma(\md)$
	\[
	\scd(Z_N)=\scd(Z_N,T)=\frac{1}{N}\sup_{A\in T}\big|D_{A}(Z_N)\big|.
	\]
	Often, $T=\{B(w,t)\ |\ w\in\md,\ t<r_\md\}$ and one speaks of the $L^\infty$--discrepancy in this case; it is also referred to as spherical cap discrepancy if $\md=\sdd$. 
\end{definition}

\begin{remark}
	Other versions of $L^p$-discrepancy exist where test sets might not be balls or certain weights are introduced, see for instance  Leobacher and Pillichshammer \cite{LeobacherPillichshammer}, Sloan and Woźniakowski \cite{SloanWozniakowski} or Stolarsky \cite{Stolarsky}.
\end{remark}

 On $[0,1)$ we let$\dm(x)=\mathrm{d}x$, while on $\s$, $\mu$ is the normalized surface measure.
The $L^p$-discrepancy of the unit $d$-dimensional hypercube, in contrast,  only averages over boxes with a corner in the origin, and all these different notions give rise to an admissible function-pair. For instance, we have $N^p\cdot\lpd(Z_N)=\mathscr{D}\big(E_{Z_N}\big)$ with
\[
\mathscr{D}(E_{Z_N})=\int_{0}^{r_\md}\int_{\md}\big| \big(E_{Z_{N}}-E_{Z_{N}}(\md)\mu\big)\big(B(w,t)\big)\big|^p\dm(w)\ \dt,
\] 
where $f(n)=n^{p-1}$ for $p\geq 1$ and $f\equiv 1$ for $0<p<1$. To get the mentioned $L^p$-discrepancy on $\md=[0,1)^d$, one uses integrals over side lengths of axis-parallel boxes (not metric balls).  Furthermore, if
\[
\mathscr{D}(E_{Z_N})=\sup_{A\in T} \big| \big(E_{Z_{N}}-E_{Z_{N}}(\md)\mu\big)(A)\big| \hspace{0.3cm}\mbox{then}\hspace{0.3cm} N\cdot \scd(Z_N)=	\mathscr{D}\big(E_{Z_N}\big).
\] 
\begin{corollary}
	The $L^\infty$--discrepancy of $\big(S^\Box_{b,d}\big)_N$  is $O\big(N^{(d-1)/d}\big)$ for $d>1$.
	 
\end{corollary}
\begin{proof}	
	We use Theorem \ref{thm_main} with $P=\id$, $\md=[0,1)^d$ and $N\cdot \scd(Z_N)=	\mathscr{D}\big(E_{Z_N}\big)$ with test-sets the axis parallel boxes with one corner at the origin. Let $A$ be an axis-parallel box with side lengths 1 in $d-1$ directions and $\epsilon$ in the remaining axis direction, i.e. volume$(A)=\mu(A)=\epsilon$, then it follows by Equation \eqref{eq_SisLatticeForAllm} that $S^\Box_{N}=b^{-m}\Z^d\cap[0,1)^d$ (with $N=b^{md}$)  satisfies $(\mu(A)=\epsilon\ra 0$)
		\[
	b^{(d-1)m}\leq	\mathscr{D}(E_{S^\Box_{N}})\leq b^{(d-1)m}
		\] 
		since any $(d-1)$-dimensional face of $[0,1)^d$ that contains the origin has $b^{(d-1)m}$-many points of the lattice set $b^{-m}\Z^d\cap[0,1)^d$ which gives the lower bound, and the upper bound follows after some thought. A similar argument applies to shifted lattice sets and implies that $h_m^\infty= b^{(d-1)m}$ and hence
	\[
	\mathscr{D}(E_{S^\Box_{N}})\leq \sum_{j=0}^n\epsilon_j b^{j(d-1)}\leq C_{b,d}b^{n(d-1)}=C_{b,d}N^{\log(b)(d-1)/\log(b^d)}
	\] 
	since $	\mathscr{D}$ is subadditive, which proves the claim.	
\end{proof}

\begin{remark}
	This corollary shows that our construction is far from optimal in the unit hypercube for $d>1$ with $P$ the identity map (see Section \ref{sec_Equidis01} for optimal values), but it is very likely that one can improve the discrepancy if $P$ is something more sophisticated.
\end{remark}
Any  $\{\omega_{N}\}_{N\in I}$ such that $\{\scd(\omega_N)\}_{N\in I}$ is a null sequence is called \textbf{asymptotically uniformly distributed}.  Fundamental results due to Beck \cite{Beck} show 
\[
cN^{-\frac{1}{2}-\frac{1}{2d}} \le \scd(\omega_N),
\hspace{0.5cm}
\mbox{and}
\hspace{0.5cm}
\scd(\omega_N^{\star}) \le CN^{-\frac{1}{2}-\frac{1}{2d}}\sqrt{\log N}
\]
for all $\omega_N$ and some  $\omega_N^{\star}$ in $\sdd$ (with positive $c,C$ and independent of $N$).
Probabilistic algorithms can construct such $\omega^\star_{N}$  with high probability (see  Alishahi and Zamani \cite{AlishahiZamani} or Beltrán, Marzo and Ortega-Cerdà \cite{BeltranMarzoOrtega}), but for no deterministic construction of point sets $\omega_{N}\subset\s$ has it been proved  that $\scd(\omega_{N})=o(N^{-1/2})$, even though candidates exists according to numerical experiments (see Aistleitner, Brauchart and Dick \cite{Aistleitner}, or Brauchart et al \cite[p.2838]{BrauchartSaffSloanWomersly}).

\vspace{0.1cm}

Yet another way to quantify uniform distribution is in terms of the Wasserstein $p$-metric $W_p(\nu,\mu)$, also known as earth-mover or Kantorovich-Rubinstein metric. We will prove bounds in a collaboration with R. Matzke  \cite{FerMatzke} based on ideas in this work, which will complement results of Brown and Steinerberger \cite{BrownSteinerberger}.

\vspace{0.1cm}

Previous works were interested in the size of the following set ($S=S^\Box_{2,1}$)
\[
\mathcal{B}(\delta,N):=\{m\in\N\ |\ m< N,\  m\cdot\scd(S_m)\leq \delta\log(m)\},
\] 
see for instance    Spiegelhofer  \cite{Spiegelhofer}, where  for $\delta=\frac{1}{100}$ an order of $\Omega\big( N^{0.056}\big)$ was established: for $g:\N\ra\R$ we say that $g(N)=\Omega(N^a)$ if  $N^a=O(g(N))$. By the very nature of our Theorem \ref{thm_main}, such bounds become much more interesting for a variety of discrepancy functions. We will thus give lower bounds for
\[
\mathcal{B}_b(\delta,N):=\{m\in\N\ |\ m< N,\  \mathrm{D}(E_{Z_m})\leq \delta\log(m)\},
\] 
where $Z$ is a sequence in $\md$, $b>1,d\geq1$ and $\mathrm{D}:\mathcal{E}_\md\ra\R$ is such that  
\begin{equation}\label{eq_mathrmDboundedMbN}
	|\mathrm{D}(E_{Z_N})|\leq C\cdot \sum_{j=0}^n\epsilon_j,\hspace{0.5cm}\mbox{for all }N=\sum_{j=0}^{n}\epsilon_j\cdot b^{dj}\mbox{ and some }C>0.
\end{equation}
By Theorem \ref{thm_ApplvdC}, the $L^p$--discrepancy of (generalized) van der Corput sequences satisfy these assumptions with $\mathrm{D}(E_{Z_N})=N \lpd^{1/p}(Z_N)$, as do Niederreiter's $(0,1)$-sequences  according to Theorem \ref{thm_ApplNiederreiter}, for all $1\leq p\leq \infty$.

\begin{theorem}\label{thm_SizeOfSetsB}
	Let $Z,C,b,d,\mathrm{D}$ satisfy Equation \eqref{eq_mathrmDboundedMbN}, then with  fixed  $\beta\geq 2$,
	\[
	\Big| \mathcal{B}_b\Big(\frac{C}{\beta\cdot  \log(b^d)},N\Big)\Big|=\Omega\Bigg(\frac{N^{h(\beta,b,d)}}{\sqrt{\log(N)}}\Bigg),\hspace{0.4cm}\mbox{for }\hspace{0.4cm}\frac{\log(N)}{\log(b^d)}\geq \nu\cdot \beta
	\]
	where $\nu$ is a big enough number independent of all variables and
	\[
	h(\beta,b,d)\cdot \log(b^d)=
	\log\Big(\beta^{\frac{1}{\beta}}\big(\frac{\beta}{\beta-1}\big)^{\frac{\beta-1}{\beta}}\Big).
	\]
\end{theorem}
\begin{remark}
	 This improves the lower bound of Spiegelhofer  to  $N^{0.05968...}$ (up to log-factors): this follows with $C=d=b-1=1$, $\beta \log(2)=100$,  $Z=S^\Box_{2,1}$ and $\mathrm{D}(E_{Z_N})=N\scd(Z_N)$.  An upper bound of the order $N^{0.183}$ is also found in \cite{Spiegelhofer}.
\end{remark}
Our approach in the proof of Theorem \ref{thm_SizeOfSetsB} encounters a combinatorial object in Equation \eqref{eq_proofBoundsBbdelta}, which we estimated rather crudely (very accurately for $b=2, d=1$, though)  and  a better bound   can be achieved by a more careful asymptotic analysis, which we omitted for sake of simplicity and brevity.

\subsection{Equidistribution on $[0,1]$}\label{sec_Equidis01}

 Weyl \cite{Weyl} started to investigate equidistributed sequences on the unit interval, usually of the type $(\alpha n)$mod 1 for $\alpha$ an irrational number. The arguably most famous sequence with low discrepancy, for which we will obtain novel bounds, is due to van der Corput \cite{vanderCorput}. For an overview on these and related questions, see the survey by Faure, Kritzer and Pillichshammer \cite{FaureKritzerPillichshammer}.
Since $S^\Box_{b,1}$ is the van der Corput sequence in base $b$, Theorem \ref{thm_main} then easily  gives upper bounds  for various discrepancies of van der Corput sequences. However, note that whenever $h_m^\infty\leq C$, the next term appears naturally in our upper bounds  with  $N=\sum_{j=0}^{n}\epsilon_j b^{j}$:
\begin{equation}\label{eq_MbN}
	M_b(N):=\min\Big\{\sum_{j=0}^n\epsilon_j,\ b(n+1)-n+1-\sum_{j=0}^n\epsilon_j  \Big\}.
\end{equation}

\begin{remark}
	If $N=\sum_{j=0}^{n}\epsilon_j b^{j}$ then the expression $\sum_{j=0}^n\epsilon_j$ is called sums of digits of $N$ in base $b$, or digit sum for short.
\end{remark}
\begin{theorem}\label{thm_ApplvdC}
	Let  $P=\id$, $\md=[0,1)$ and $S=S^\Box_{b,1}$ be as  in Section \ref{sec_SboxSequence}. With  $n=\lfloor\log(N)/\log(b)\rfloor$, $N=\sum_{j=0}^n\epsilon_jb^{j}$ and $M_b(N)$ as in Equation \eqref{eq_MbN}, we obtain 
	\[
		N\cdot	\sqrt[p]{\lpd(S_N)}\ \leq\ 
			\begin{dcases}
				\sqrt[p]{\frac{M_b(N)}{p+1}} & \text{if } 0<p< 1,\\
						\frac{M_b(N)}{\sqrt[p]{p+1}}	&  \text{if } 1\leq p<\infty,
				\end{dcases}
	\]
	and for $p=\infty$ we have $N\cdot	\scd(S_N,T)\ \leq \ M_b(N)$ for $T=\{[0,t]\ |\ 0<t<1\}$.
\end{theorem}
\begin{remark}
	For $b=2$, $p=\infty$ we have $M_b(N)\leq n/2+3/2\leq \onehalf \log(N)/\log(2)+3/2$, which is half the upper bound as obtained in van der Corput's original result, and  close to the optimal value $\frac{1}{3}\log(N)/\log(2)+1$  due to Béjian and Faure \cite{BejianFaure}. 
	For $b=p=2$  we have an upper bound of $ \log(N)/\big(\log(2)2\sqrt{3}\big)+\sqrt{3}/2$, which again is close to the optimal value (for any $p\geq 1$) of $ \log(N)/\big(\log(2)6\big)+1$ (see Pillichshammer \cite{Pillichshammer}). 
	The case $0<p<1$ is completely new for any $b>1$  and the only other work we know that gives bounds for some point sets and  $p$ in this range is by Skriganov \cite{Skriganov}.
\end{remark}

The optimal order for any sequence and $1\leq p<\infty$ is  $\sqrt{\log(N)}$ (see \cite{FaureKritzerPillichshammer}).  
 Faure introduced a generalized version of van der Corput sequences in \cite[p.769]{FaureKritzerPillichshammer}, which is covered by our definition of $S^\Box_{b,1}$ too. Thus Theorem \ref{thm_ApplvdC} also applies to these sequences, and the built-in symmetry of $M_b(N)$ gives much better bounds for $N$ with small/large digit sums. Such symmetric behavior for $\scd$ is known (see  \cite[Fig.3]{FaureKritzerPillichshammer}) but we are not aware of any result remotely comparable to ours. Digit sums also appear in a  work of Beretti \cite{Beretti} with a different discrepancy function, which is also covered by our methods.

\begin{theorem}\label{thm_ApplNiederreiter}
	With the same setup as in  Theorem \ref{thm_ApplvdC}, but with $S=S^\boxplus_{b,1}$  as constructed in Section \ref{subsec_GuidedSequences},	we obtain 
	\[
	N\cdot	\sqrt[p]{\lpd(S_N)}\ \leq\ 
	\begin{dcases}
		2\cdot\sqrt[p]{\frac{M_b(N)}{p+1}} & \text{if } 0<p< 1,\\
		2\cdot\frac{M_b(N)}{\sqrt[p]{p+1}}	&  \text{if } 1\leq p<\infty,
	\end{dcases}
	\]
	and for $p=\infty$ we have $N\cdot	\scd(S_N,T)\ \leq \ 2M_b(N)$ for $T=\{[0,t]\ |\ 0<t<1\}$.
\end{theorem}

 One could also obtain similar bounds for the  \textit{diaphony} of our 1--dimensional sequences, which can be written as a double integral (see Faure \cite[p.127]{Faure}).

\subsection{Equidistribution on $\s$}

The problem of distributing  points uniformly on a sphere has applications in numerical integration, approximation,  and the applied sciences (see for instance \cite{Bauer,BrauchartGrabner,Gorski,KuijlaarsSaff}). There are many constructions that produce sequences of point sets $\{\omega_{N}\}_{N\in I}$, for $I\subset \N$ with $|I|=\infty$ and  $\omega_N\subset\sdd$, to satisfy some desirable properties like  provably small bounds on their energies (Riesz, Logarithmic) or their discrepancies (see  \cite{Aistleitner,AlishahiZamani,BeltranEtayo,BeltranEtayoLopez,Etayo,Ferizovic,FerHofMas,Hardin,LubotzykPhillipsSarnak}).

 \vspace{0.2cm}
 
 For such constructions  and two different $N,N'\in I$, one usually finds that $\omega_{N}\cap \omega_{N'}=\emptyset$, which  makes them  computationally expensive if point evaluations are involved. Sequences, on the other hand, do not suffer this drawback and should prove useful in numerical applications. Yet,  works that \textit{prove} results for  sequences on $\sdd$ seem  rare. One such result is given below.
 
 \vspace{0.2cm}
 
 \textbf{Theorem A} (\cite{ArnoldKrylov} Arnol'd \& Krylov '63)  Given two rotations $A,B$ and some $x\in\s$, let $Z=(x,Ax,Bx,A^2x,ABx,BAx,B^2x,\ldots)$. If the set induced by $Z$ is dense, then the sequence $Z$  is asymptotically uniformly distributed.
 
 \vspace{0.2cm}
 Furthermore, around the 1950's, researchers like Edrei, Leja, and G\'{o}rski investigated sequences that were constructed via an optimization strategy and this greedy construction was applied on spheres in the work by López-García \& McCleary \cite{LopGarciaMcCleary}, where the reader can find more background information. 
%
 
  \vspace{0.2cm}
 
 On $\s$, we can bound the $L^\infty$--discrepancy of our sequences, which are much easier and much faster to generate than greedy sequences:

 \begin{theorem}\label{thm_AppS2}
 		Fix $b\in\N_{>1}$ and let $\mathscr{Z}=\big\{ L(S^\boxplus_{b,2})\ |\ S^\boxplus_{b,2}\mbox{ as in Section \ref{subsec_GuidedSequences}} \big\}$ where  $L$ is the Lambert cylindrical equal-area projection. Then for any $Z\in\mathscr{Z}$ we have, with some constant $C'=C'(b)$ independent of $N=\sum_{j=0}^{n}\epsilon_j\cdot b^{2j}$,
 		\begin{align*}
 			\scd(Z_N)\ &\leq\ \frac{2+3\sqrt{2}}{N}
 			\min\Big\{\sum_{j=0}^n\epsilon_jb^j,\ b^{n+1}(b+2)-\sum_{j=0}^n\epsilon_jb^j  \Big\} +\frac{C'}{N}\log(N)\\
 			&\leq\ \frac{2+3\sqrt{2}}{2} \frac{b(b+2)}{\sqrt{N}} +\frac{C'}{N}\log(N).
 		\end{align*}
 \end{theorem}

Theorem \ref{thm_main} also  applies if the domain of $P$ is a union of copies of $[0,1)^d$, as is the case with the HEALPix map, see \cite{FerHofMas,Gorski} which is a building block to construct the map $G$ in the next theorem.

\begin{theorem}\label{thm_AppHealpix}
		Let $\mathscr{Z}=\big\{ G(S^\Box_{2,2})\ |\ S^\Box_{2,2}\mbox{ as in Section \ref{sec_SboxSequence} } \big\}$ where  $G$ is as in Equation \eqref{eq_defGprojection}. Then for any $Z\in\mathscr{Z}$ we have, with some constant $C'$ independent of $N=12\cdot K+k$, $0\leq k<12$,  $\kappa=\lfloor\log(K)/\log(4)\rfloor$,  and  $K=\sum_{j=0}^{\kappa}\epsilon_j\cdot 2^{2j}$
			\begin{align*}
			\scd(Z_N)\ &\leq\ 4 \frac{5+\sqrt{2}}{N}
			\min\Big\{\sum_{j=0}^\kappa\epsilon_j2^j,\ 2^{\kappa+3}-\sum_{j=0}^\kappa\epsilon_j2^j  \Big\} +\frac{C'}{N}\log(N)\\
			&\leq\  \frac{8}{\sqrt{3}}\frac{5+\sqrt{2}}{\sqrt{N}} +\frac{C'}{N}\log(N).
		\end{align*}
\end{theorem}

\section{Fundamental Sequences and Theorem \ref{thm_main} }\label{sec_Theorem1SBox}

We will next present an algorithm to construct many  sequences  in $[0,1)^d$ which  form the backbone of this work and as we recently learned appear implicitly in the work of Infusino and Volčič \cite{InfusinoVolcic}. In Figure \ref{fig:Thm1Visualization} we present a visual proof for Theorem \ref{thm_main}.

\subsection{The Sequences $S^{\Box}_{b,d}$}\label{sec_SboxSequence}

 We will  index a sequence in $[0,1)^d$ by starting with 0, thus $S=(s_0,s_1,\ldots)$ and $S_N=(s_0,\ldots,s_{N-1})$.
\begin{definition}\label{def_sbox}
	The sequence $S=S^{\Box}_{b,d}$ in base $b\in\N$, with $b>1$, is defined recursively: Set $s_0=\vec{0}$. Suppose the first $b^{dm}$ elements have been defined, then $S_{b^{d(m+1)}}$ is obtained by concatenating $b^d-1$ shifted versions of $S_{b^{dm}}$ to itself as follows: let  $\{H_1,\ldots,H_{b^d-1}\}=\{0,\ldots,b-1\}^d\setminus\{\vec{0}\}$, let $(\pi_1,\pi_2,\ldots)$ be a sequence of permutations of $(1,\ldots,b^d-1)$ and set
	\[
	\big(s_{k\cdot b^{dm}},\ldots,s_{(k+1)\cdot b^{dm}-1}\big):=
	b^{-m-1}H_{\pi_m(k)}+\big(s_{0},\ldots,s_{b^{dm}-1}\big),
	\]
	for $k\in\{1,2,,\ldots,b^d-1\}$.
\end{definition}
\begin{remark}
	Adding $b^{-m-1}H_{\pi_m(k)}$ to $S_{b^{dm}}$ is done element-wise. 
\end{remark}
We will use the short-hand $S^{\Box}=S^{\Box}_{b,d}$ and note that
\begin{equation}\label{eq_SisLatticeForAllm}
	S^\Box_{b^{dm}}=\ell^d_{b^m}=\big(b^{-m}\Z^d\big)\cap [0,1)^d\hspace{0.5cm}\mbox{for all }m\in\N_0, \mbox{ and all }(\pi_1,\pi_2,\ldots).
\end{equation}
\begin{figure}[h!]
	\begin{center}
		\includegraphics[scale=0.36]{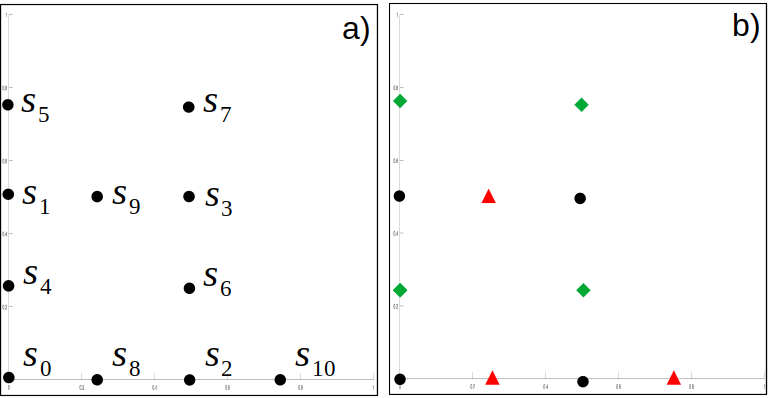}
		\caption{ The first 11 elements of  a possible choice for $S^{\Box}_{2,2}$ in \textrm{a}). In \textrm{b}) we see the underlying construction via shifting $S_{2^2}$ clearly, and its order is seen in \textrm{a}).}
		\label{fig:Sbox}
	\end{center}
\end{figure}
\begin{remark}
	 If $d=1$ and $\pi_m=\fatone$ the identity, we obtain  van der Corput's sequence in base $b$: a proof for the bit-reversal  in case $b=2$ follows  from \cite[Equ.(23)]{Wang} - note that in the paper's notation $J_4=(0,2,1,3)$ refers in our case to $s_0<s_2<s_1<s_3$, and $s_n$ is the bit-reversal of $n$ in base (radix) 2. 
	 	For other $b$ we refer to a small paragraph of page 52 in \cite{InfusinoVolcic}. General $\pi_m$ yield  the generalized version of van der Corput sequences of   \cite[p.769]{FaureKritzerPillichshammer}.
\end{remark}

\begin{remark}
	In our eyes this algorithm is very natural and it so happens that we have rediscovered it.  Infusino and  Volčič give a construction for fractal sets which also applies to $[0,1)^d$, see Section 2 of \cite{InfusinoVolcic}. They don't consider sequences of permutations as we do and they compute discrepancies with restricted test sets coming from their construction, but then again their algorithm is very general and also applies to triangles as done by Basu \& Owen in \cite{BasuOwen}.
\end{remark}

The next lemma stresses the fractal structure of  $S^\Box$ and explains the appearance of  the sets $W^{m,b}_t$ in Theorem \ref{thm_main}. We recall the convention that  $\sum_{j=t}^{n}a_j:=0$ if $n<t$ and $a_j$ arbitrary.

\begin{lemma}\label{lem_latticeTracing}
	Let  $m\in\N$ and $(\varepsilon_1,\ldots,\varepsilon_m)\in\{0,1,\ldots,b^d-1\}^m$  be chosen arbitrarily.
	For $t\in\N$ set 
	\[
				K(t)=	K_{m,\varepsilon_{j}}(t)=\sum_{r=t}^m\varepsilon_r b^{dr}.
	\]
%
	Then there is a vector $v_{m,\varepsilon_{j}, \pi_j}(t)\in[0,b^{-t})^d\cap \big(\Z^d\cdot b^{-m-1}\big)$,  such that for  $k<b^{dt}$ 
	\[
	\big(s_{K(t)},\ldots,s_{K(t)+k}\big)=v_m(t)+ \big(s_{0},\ldots,s_{k}\big)=v_m(t)+ S^\Box_{k+1}.
	\]
\end{lemma}
\begin{proof}  By Definition  \ref{def_sbox}, and in its notation with $H_{\pi_j(0)}:=\vec{0}$, it follows that for $k'<b^{dm}$
	\[
	\big(s_{\varepsilon_m\cdot b^{dm}},\ldots,s_{\varepsilon_m\cdot b^{dm}+k'}\big)=\frac{H_{\pi_m(\varepsilon_m)}}{b^{m+1}}+	\big(s_{0},\ldots,s_{k'}\big),
	\]
	which covers the case $t=m$. If $t=m-1$ and $k''<b^{dm-d}$, then  with  $k'=\varepsilon_{m-1}\cdot b^{dm-d}+k''$ we obtain from the previous case that
	\begin{align*}
		&\big(s_{\varepsilon_m\cdot b^{dm}+\varepsilon_{m-1}\cdot b^{dm-d}},\ldots,s_{\varepsilon_m\cdot b^{dm}+\varepsilon_{m-1}\cdot b^{dm-d}+k''}\big)\\
		&\hspace{4
			cm}=\frac{H_{\pi_m(\varepsilon_m)}}{b^{m+1}}+	 \big(s_{\varepsilon_{m-1}\cdot b^{dm-d}},\ldots,s_{\varepsilon_{m-1}\cdot b^{dm-d}+k''}\big)\\
		&\hspace{4
			cm}=\frac{H_{\pi_m(\varepsilon_m)}}{b^{m+1}}+\frac{H_{\pi_{m-1}(\varepsilon_{m-1})}}{b^{m}}+	 \big(s_{0},\ldots,s_{k''}\big),
	\end{align*}
	where we  just used the final $k''+1$ elements. Similarly, for any $k<b^{dt}$ 
	\[
	\big(s_{K(t)},\ldots,s_{K(t)+k}\big)=\frac{H_{\pi_m(\varepsilon_m)}}{b^{m+1}}+ \big(s_{K(t)-\varepsilon_m\cdot b^{dm}},\ldots,s_{K(t)-\varepsilon_m\cdot b^{dm}+k}\big).
	\]
	We then obtain by induction 
	\[
	\big(s_{K(t)},\ldots,s_{K(t)+k}\big)=\sum_{j=0}^{m-t}\frac{H_{\pi_{m-j}(\varepsilon_{m-j})}}{b^{m+1-j}}+ \big(s_{0},\ldots,s_{k}\big).\qedhere
	\]
\end{proof}
\begin{figure}[h!]
	\begin{center}
		\includegraphics[scale=0.39]{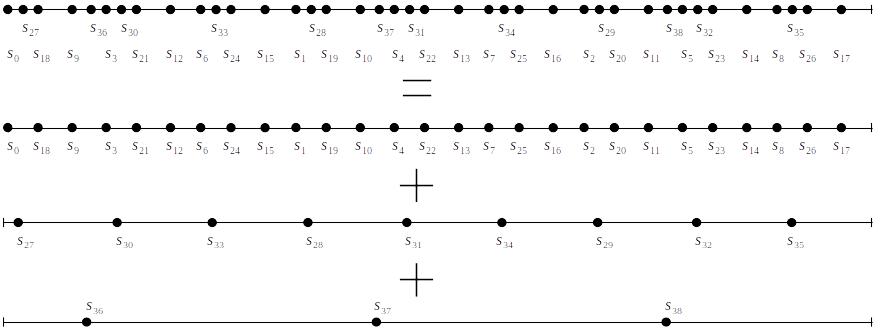}
		\caption{ Here we see how Lemma \ref{lem_latticeTracing} works for $N=39=3^3+3^2+3$ and  $S=S^\Box_{3,1}$ with $\pi_1=\pi_2=\pi_4$  the identity permutation of elements $(1,2)$, and $\pi_3\neq \pi_1$. }
		\label{fig:Lem8Visualization}
	\end{center}
\end{figure}

As indicated in Figure \ref{fig:Lem8Visualization}, Lemma \ref{lem_latticeTracing} and Equation \eqref{eq_SisLatticeForAllm} imply for  all $t<m$
\begin{equation}\label{eq_SisShiftedLatticeForAllm}
	\big(s_{K(t)},\ldots,s_{K(t)+b^{dt}-1}\big)=v+ S^\Box_{b^{dt}}=v+\ell^d_{b^t}=
\big(v+b^{-t}\Z^d\big)\cap [0,1)^d.
\end{equation}

\subsection{Proof of Theorem \ref{thm_main} with $L_m$ being lattice sets}\label{subsec_proofThm1LmLatticeSets}

Recall that  $\sum_{j=t}^{m}a_j:=0$ if $m<t$ and set $z_{j+1}=P(s_j)$ for $j\in\N_0$. We henceforth consider the sequence $Z=(z_1,\ldots)$, starting with the index 1. 

 \textbf{We first prove Equation \eqref{eq_mainResultN}.}
	Let  $N=\sum_{j=0}^n \epsilon_jb^{dj}$ with $0\leq \epsilon_j<b^d$, $\epsilon_n>0$. Equation \eqref{eq_SisShiftedLatticeForAllm} implies that $Z_N=\bigcup_{K}P(L'_{K})$ for some shifted lattice sets $L'_{K}$:
	\begin{equation*}
		\sum_{j=1}^{N}
		\delta(z_j)= \sum_{j=0}^n\sum_{k=1}^{\epsilon_{n-j}} \sum_{r=K(j,k)}^{K(j,k+1)-1}
		\delta(z_r),
	\end{equation*}
	where
	\begin{equation*}
		K(j,k)=1+(k-1)\cdot b^{d(n-j)}+\sum_{r=n-j+1}^{n}\epsilon_{r}b^{dr}.
	\end{equation*}	
	By Lemma \ref{lem_latticeTracing} and Equation \eqref{eq_SisShiftedLatticeForAllm}, there is  a  $v\in[0,b^{j-n})^d\cap \big(\Z^d\cdot b^{-n-1}\big)$ such that 
	\begin{align*}
		\big\{z_{K(j,k)},\ldots,z_{K(j,k+1)-1}\big\}&=P\big(\big\{s_{K(j,k)-1},\ldots,s_{K(j,k+1)-2}\big\}\big)\\
		&=P\big(\big[b^{j-n}\cdot\Z^d+v\big]\cap [0,1)^d\big)=:W_{n-j}.
	\end{align*}
	By our assumptions, $\mathscr{D}\big(E_{W_{n-j}}\big)\leq h_{n-j}^n$, and thus, by $f$--subadditivity,
	\allowdisplaybreaks
	\begin{align*}
	\mathscr{D}\big(E_{Z_N}\big)&=\mathscr{D}\Bigg(\sum_{j=0}^n\sum_{k=1}^{\epsilon_{n-j}} \sum_{r=K(j,k)}^{K(j,k+1)-1}
	\delta(z_r)\Bigg)\\
	&\leq f\Big(\sum_{j=0}^n\epsilon_j\Big)\cdot\sum_{j=0}^n\sum_{k=1}^{\epsilon_{n-j}}\mathscr{D}\Bigg( \sum_{r=K(j,k)}^{K(j,k+1)-1}
	\delta(z_r)\Bigg)\\
	&\leq f\Big(\sum_{j=0}^n\epsilon_j\Big)\cdot\sum_{j=0}^n\sum_{k=1}^{\epsilon_{n-j}}h_{n-j}^n=
	f\Big(\sum_{j=0}^n\epsilon_j\Big)\cdot\sum_{j=0}^n\epsilon_{j}h_{j}^n. 
	\end{align*}

\textbf{Next we prove Equation \eqref{eq_mainResultNprime}.}
We define index boundaries $G(j,k)$ to reflect the construction of points $S^\Box_{b,d}$:
\[
G(j,k)=G_{n}(j,k):=1+\big(k+\epsilon_{j}\big)\cdot b^{dj}+\sum_{r=j+1}^{n}\epsilon_{r}b^{dr}.
\]
Then we can write
\begin{align*}
		\sum_{j=1}^{N}
	\delta(z_j)&=	\sum_{j=1}^{b^{d(n+1)}}
	\delta(z_j)-	\sum_{j=N+1}^{N+b^d-\epsilon_0}
	\delta(z_j)-	\sum_{j=N+b^d-\epsilon_0+1}^{b^{d(n+1)}}
	\delta(z_j)\\
	&=	\sum_{j=1}^{b^{d(n+1)}}	\delta(z_j)-	\sum_{j=N+1}^{N+b^d-\epsilon_0}
	\delta(z_j)- \sum_{j=1}^n\sum_{k=1}^{b^{d}-1-\epsilon_j} \sum_{r=G(j,k)}^{G(j,k+1)-1}
	\delta(z_r).
\end{align*}
Note that $G(j,k+1)=G(j,k)+b^{dj}$ and hence the rightmost sum above adds $b^{dj}$ terms.	By Lemma \ref{lem_latticeTracing} and Equation \eqref{eq_SisShiftedLatticeForAllm} (with $\varepsilon_j=\epsilon_{j}+k$ and $\varepsilon_s=\epsilon_{s}$ for $s> j$), there is  a  $v\in[0,b^{-j})^d\cap \big(\Z^d\cdot b^{-n-1}\big)$ such that 
\begin{align*}
\big\{z_{G(j,k)},\ldots,z_{G(j,k+1)-1}\big\}&=P\big(\big\{s_{G(j,k)-1},\ldots,s_{G(j,k+1)-2}\big\}\big)\\
	&=P\big(\big[b^{-j}\cdot\Z^d+v\big]\cap [0,1)^d\big)=:W_{j}',
\end{align*}
and similarly for $N<r\leq N+b^d-\epsilon_0$ there is  a  $v\in[0,1)^d\cap \big(\Z^d\cdot b^{-n-1}\big)$ 
\begin{equation*}
z_{r}=P\big(s_{r-1}\big)=P\big(\big[\Z^d+v\big]\cap [0,1)^d\big).
\end{equation*}
By our assumptions, $\mathscr{D}\big(E_{W_{j}'}\big)\leq h_{j}^{n}$ and $\mathscr{D}\big(\delta(z_r)\big)\leq h_{0}^{n}$, hence, by $f$--subadditivity,
\allowdisplaybreaks
\begin{align*}
	&\mathscr{D}\big(E_{Z_N}\big)=\mathscr{D}\Bigg[	\sum_{j=1}^{b^{d(n+1)}}	\delta(z_j)-	\sum_{j=N+1}^{N+b^d-\epsilon_0}
	\delta(z_j)- \sum_{j=1}^n\sum_{k=1}^{b^{d}-1-\epsilon_j} \sum_{r=G(j,k)}^{G(j,k+1)-1}
	\delta(z_r)\Bigg]\\
	&\leq f\Big(1-n+\sum_{j=0}^n\big(b^{d}-\epsilon_j\big)\Big)\Bigg[\mathscr{D}\big(E_{Z_{b^{d(n+1)}}}\big)+(b^d-\epsilon_0)h_{0}^{n}+\sum_{j,k}\mathscr{D}\Big( \sum_{r}
	\delta(z_r)\Big)\Bigg]\\
	&\leq f\Big(1-n+\sum_{j=0}^n\big(b^{d}-\epsilon_j\big)\Big)\Bigg[h_{n+1}^{n}+h_{0}^{n}+\sum_{j=0}^n\sum_{k=1}^{b^{d}-1-\epsilon_j}h_j^{n}\Bigg]. \hspace{2cm} \qed
\end{align*}

\begin{figure}[h!]
	\begin{center}
		\includegraphics[scale=0.36]{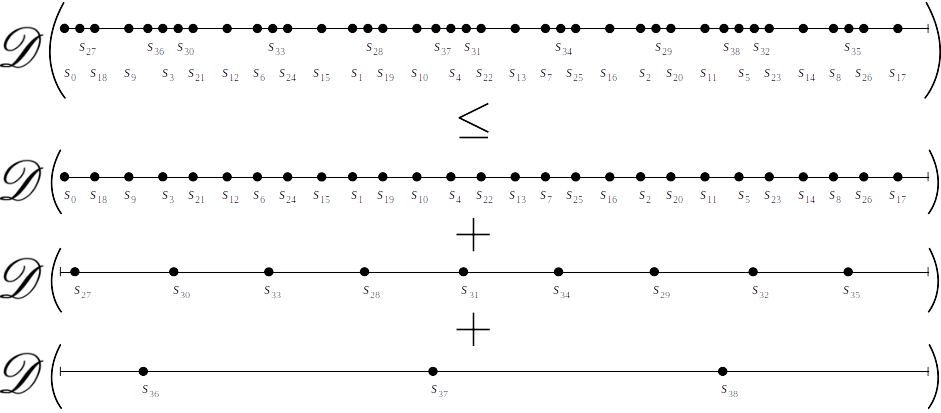}
		\caption{  Let $S_N$ be like in Figure \ref{fig:Lem8Visualization} and $\mathscr{D}$ be 1-subadditive. Here we see how Lemma \ref{lem_latticeTracing} and 1-subadditivity work hand in hand to bound $\mathscr{D}(E_{S_N})$. }
		\label{fig:Thm1Visualization}
	\end{center}
\end{figure}

\section{Guided Sequences and Theorem \ref{thm_main} }

We will next present an algorithm to construct many sequences $S^{\boxplus}=S^{\boxplus}_{b,d}$ on $[0,1)^d$ which use $S^\Box$ as a "guide". For $d=1$,  we recover $(0,1)$--sequences as introduced by Niederreiter  in \cite{Niederreiter} (see also \cite[p.792]{FaureKritzerPillichshammer}).

\subsection{The $S^\Box$--guided sequences $S^{\boxplus}$}\label{subsec_GuidedSequences}

Loosely speaking, an $S^\Box$--guided sequence is defined recursively  to fill out empty boxes of $\ell^d_{b^m}=\big(b^{-m}\Z^d\big)\cap [0,1)^d$, where the order in which this is done is guided by the fundamental sequence $S^\Box=S^{\Box}_{b,d}=(s_0,s_1,\ldots)$; recall Equation \eqref{eq_SisLatticeForAllm}.

\begin{definition}\label{def_LatticeGuided}
		For  arbitrary but fixed integers $b>1$ and $d\geq 1$ and a given  sequence $S^\Box=S^\Box_{b,d}=\big(s_0,s_1,\ldots\big)$  as in Definition \ref{def_sbox}, we call a sequence $S^{\boxplus}_{b,d}=\big(q_0,q_1,\ldots\big)$  \textbf{$S^\Box$-guided}, if $q_0\in[0,1)^d$ is arbitrary, and for any integer $m\geq 0$ the elements
		\[
		\big(q_{b^{dm}},\ldots,q_{b^{d(m+1)}-1}\big)
		\]
		are  such that $\{q_0,\ldots q_{b^{d(m+1)}-1}\}$ satisfies the following two conditions:
		\begin{enumerate}
			\item (lattice filling) for all $v\in \big(b^{-m-1}\Z^d\big)\cap [0,1)^d$
			\[
			\Big|\{q_0,\ldots q_{b^{d(m+1)}-1}\}\cap\big(b^{-(m+1)}[0,1)^d+v\big)\Big|=1;
			\]
			\item  ($S^\Box$-guidance) for every $0\leq r<b^{dm}$ and $1\leq k<b^d$, we have
			\[
			q_{r+k\cdot b^{dm}}\in \big(s_r+b^{-m}[0,1)^d\big).
			\]
		\end{enumerate}
\end{definition}
\begin{figure}[h!]
	\begin{center}
		\includegraphics[scale=0.35]{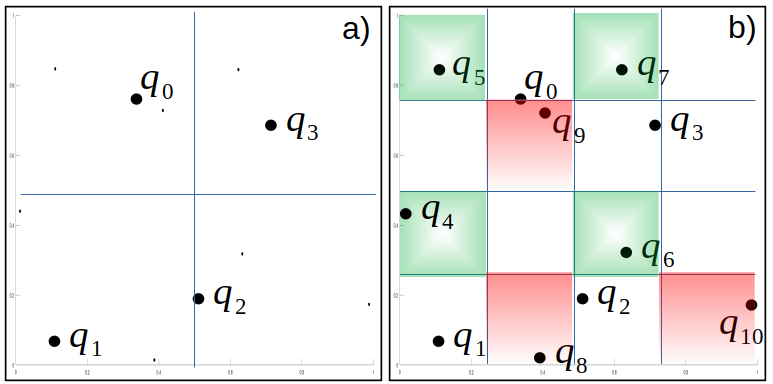}
		\caption{ We see the first 11 elements of  a possible choice for $S^{\boxplus}_{2,2}$, and how  lattice filling and the $S^\Box$-guidance  conditions are satisfied, where $S^\Box_{2,2}$ is as in Figure \ref{fig:Sbox}, which is also the base for the color scheme. The point $q_{11}$ could be placed  in one of the boxes to the right of $q_3$ or $q_7$; the point $ q_{12}$ must be placed in the box between $q_9$ and $q_8$; $q_{13}$ must be placed between $q_5$ and $q_4$, etc.}
		\label{fig:Sboxplus}
	\end{center}
\end{figure}
\begin{remark}
	 We note that in Definition \ref{def_LatticeGuided}, the set $\big(s_r+b^{-m}[0,1)^d\big)$ contains exactly one element $q_t$ for some $t<b^{dm}$. It is not necessarily always possible to also have $q_{r}\in \big(s_r+b^{-m}[0,1)^d\big)$ for all $r<b^{dm}$, but this can be achieved at least $(b^d-2)^m$-many times, see Figure \ref{fig:Sboxplus}.  The sequence $S^\Box$ is trivially  $S^\Box$--guided.
\end{remark}
The next corollary  shows that the "tails" of $S^\Box$ and $S^\boxplus$ approach each other.

\begin{corollary}
	Let $S^\boxplus_{b,d}$ be $S^\Box_{b,d}$--guided and $m(t)=\lfloor \log(t)/\log(b^d)\rfloor$, then for $t>0$
	\[
	\| s_t-q_t\|_{\infty}\leq b^{-m(t)}\leq b\cdot t^{-1/d}.
	\]
\end{corollary}
\begin{proof}
	 Clearly $b^{dm(t)}\leq t=r+k\cdot b^{dm(t)}$ with $0\leq r<b^{dm(t)}$ and $1\leq k<b^d$. We have $s_t,q_t\in \big(s_r+b^{-m(t)}[0,1)^d\big)$ by Definition \ref{def_sbox} and Definition \ref{def_LatticeGuided} respectively. 
\end{proof}

\begin{lemma}\label{lem_latticeTracingGuided}

	Let  $m\in\N$, $(\varepsilon_1,\ldots,\varepsilon_m)$ and $K(t)$ be as in Lemma  \ref{lem_latticeTracing} and $S^{\boxplus}_{b,d}=(q_0,q_1,\ldots)$ be an  $S^\Box_{b,d}$--guided sequence. For any $0< t\leq m$, there is a vector $v=v_{m,\varepsilon_{j},\pi_j}(t)\in[0,b^{-t})^d\cap \big(\Z^d\cdot b^{-m-1}\big)$ such that
	\[
	\big\{q_{K(t)},\ldots,q_{K(t)+b^{dt}-1}\big\}-v
	\]
	is a perturbed lattice set of resolution $b^{-t}$, i.e. 
	 for all $w\in \ell^d_{b^t}$ we have
	 \[
	 \Big|\Big(\big\{q_{K(t)},\ldots,q_{K(t)+b^{dt}-1}\big\}-v\Big)\cap\big(w+b^{-t}[0,1)^d\big)\Big|=1.
	 \]
\end{lemma}
\begin{proof}
	Without loss of generality  we assume $\varepsilon_{m}>0$ and set
		\[
	K(t)=	\varepsilon_tb^{dt}+\sum_{r=t+1}^m\varepsilon_r b^{dr}.
	\]
	By the $S^\Box$-guidance condition of Definition \ref{def_LatticeGuided}, we have for the choice $k=\varepsilon_{m}$ and
	\[
	r\in\Big\{K(t)-\varepsilon_mb^{dm},\ldots,K(t)-\varepsilon_mb^{dm}+b^{dt}-1\Big\}
	\] 
	that 
	\[
	q_{r+\varepsilon_{m}\cdot b^{dm}}\in \big(s_r+b^{-m}[0,1)^d\big).
	\]
	If $t=m$, then the claim of the lemma is satisfied with $v=0$ since $S^\Box_{b^{dt}}=\ell^d_{b^t}$ by Equation \eqref{eq_SisLatticeForAllm}. If $t<m$, then, for some  $v(t)\in[0,b^{-t})^d\cap \big(\Z^d\cdot b^{-m-1}\big)$  by Lemma \ref{lem_latticeTracing}, 
	\[
\big(s_{K(t)-\varepsilon_mb^{dm}},\ldots,s_{K(t)-\varepsilon_mb^{dm}+b^{dt}-1}\big)=v(t)+S^\Box_{b^{dt}},
	\]
	or equivalently $s_r=s_{u+K(t)-\varepsilon_mb^{dm}}=v(t)+s_{u}$ for some $0\leq u<b^{dt}$, and hence
	\[
	q_{r+\varepsilon_{m}\cdot b^{dm}}\in
	\big(s_{r}+b^{-m}[0,1)^d\big)
		\subset \big(v(t)+s_{u}+b^{-t}[0,1)^d\big).\qedhere
	\]
\end{proof}

\subsection{Proof of Theorem \ref{thm_main}  with $L_m$ being perturbed lattices}\label{subsec_proofThm1LmPerturbed}

Recall that  $\sum_{j=t}^{m}a_j:=0$ if $m<t$ and set $z_{j+1}=P(q_j)$ for $j\in\N_0$ and $S^\boxplus=S^{\boxplus}_{b,d}=\big(q_0,q_1,\ldots\big)\subset [0,1)^d$. We henceforth consider the sequence $Z=(z_1,\ldots)$. 

\textbf{We first prove Equation \eqref{eq_mainResultN}.}
Let  $N=\sum_{j=0}^n \epsilon_jb^{dj}$ and recall the definition of $K(j,k)$ of 
Section \ref{subsec_proofThm1LmLatticeSets} such that we can write 
\begin{equation*}
	\sum_{j=1}^{N}
	\delta(z_j)= \sum_{j=0}^n\sum_{k=1}^{\epsilon_{n-j}} \sum_{r=K(j,k)}^{K(j,k+1)-1}
	\delta(z_r).
\end{equation*}
By Lemma \ref{lem_latticeTracingGuided}, there is vector $v$ and a perturbed lattice set $\omega_{b^{d(n-j)}}$ such that 
\begin{align*}
	\big\{z_{K(j,k)},\ldots,z_{K(j,k+1)-1}\big\}&=P\big(\big\{q_{K(j,k)-1},\ldots,q_{K(j,k+1)-2}\big\}\big)\\
	&=P\big(T_v(\omega_{b^{d(n-j)}})\big)=:W_{n-j}.
\end{align*}
By our assumptions, $\mathscr{D}\big(E_{W_{n-j}}\big)\leq h_{n-j}^n$, and  the same argument as in Section \ref{subsec_proofThm1LmLatticeSets} shows
\[
\mathscr{D}\big(E_{Z_N}\big)\leq 
f\Big(\sum_{j=0}^n\epsilon_j\Big)\cdot\sum_{j=0}^n\epsilon_{j}h_{j}^n. 
\]
\textbf{Next we prove Equation \eqref{eq_mainResultNprime}.}
Recall the definition of $G(j,k)$ in Section \ref{subsec_proofThm1LmLatticeSets}:
\begin{equation*}
	\sum_{j=1}^{N}
	\delta(z_j)=	\sum_{j=1}^{b^{d(n+1)}}	\delta(z_j)-	\sum_{j=N+1}^{N+b^d-\epsilon_0}
	\delta(z_j)- \sum_{j=1}^n\sum_{k=1}^{b^{d}-1-\epsilon_j} \sum_{r=G(j,k)}^{G(j,k+1)-1}
	\delta(z_r).
\end{equation*}
By Lemma \ref{lem_latticeTracingGuided}, there is vector $v$ and a perturbed lattice set $\omega_{b^{dj}}$ such that 
\begin{align*}
	\big\{z_{G(j,k)},\ldots,z_{G(j,k+1)-1}\big\}&=P\big(\big\{q_{G(j,k)-1},\ldots,q_{G(j,k+1)-2}\big\}\big)\\
	&=P\big(T_v(\omega_{b^{dj}})\big)=:W_{j}',
\end{align*}
and similarly for $N<r\leq N+b^d-\epsilon_0$ there is  a  $v\in[0,1)^d\cap \big(\Z^d\cdot b^{-n-1}\big)$ 
\begin{equation*}
	z_{r}=P\big(q_{r-1}\big)=P\big(\big[\Z^d+v\big]\cap [0,1)^d\big).
\end{equation*}
By our assumptions, $\mathscr{D}\big(E_{W_{j}'}\big)\leq h_{j}^{n}$ and $\mathscr{D}\big(\delta(z_r)\big)\leq h_{0}^{n}$, hence, by $f$--subadditivity, the same argument as in Section \ref{subsec_proofThm1LmLatticeSets} shows
\[
\mathscr{D}\big(E_{Z_N}\big)\leq f\Big(1-n+\sum_{j=0}^n\big(b^{d}-\epsilon_j\big)\Big)\Bigg[h_{n+1}^{n}+h_{0}^{n}+\sum_{j=0}^n\sum_{k=1}^{b^{d}-1-\epsilon_j}h_j^{n}\Bigg].  \qed
\]

\section{Proof of Theorem \ref{thm_SizeOfSetsB}}\label{sec_proofSizeofSetsB}

	Let $\delta=C\cdot\big(\beta\cdot  \log(b^d)\big)^{-1}$  with $\beta\geq 2$ fixed. We can write any $m\in\N$ as $m=\sum_{j=0}^{l(m)}\epsilon_j(m)b^{dj}$ for $\epsilon_{j}(m)\in\{0,1,\ldots,b^d-1\}$  and $l(m)=\lfloor\log(m)/\log(b^d)\rfloor$.
	We will give a lower bound on a subset of numbers $m<N$  that satisfy
	\begin{equation}\label{eq_proofBoundsBbdelta}
		\sum_{j=0}^{l(m)}\epsilon_j(m)\ \leq\ \frac{l(m)}{\beta},
	\end{equation}
	which then gives a lower bound on $\big| \mathcal{B}_b(\delta,N)\big|$ by the following chain of inequalities:
	\[
		|\mathrm{D}(E_{Z_m})|\  \leq\  C	\sum_{j=0}^{l(m)}\epsilon_j(m)\ \leq\ C	\frac{l(m)}{\beta}\ \leq\ \delta\cdot\log(b^d)\cdot l(m)\ \leq\ \delta\log(m).
	\]
	For $t\in\N$ we define $J(t):=[t\beta,(t+1)\beta) \cap \N$.
 Then we have the following equivalence concerning Equation \eqref{eq_proofBoundsBbdelta} if $l(m)\in J(t)$:
 \begin{equation}\label{eq_Thm3boundbyt}
 	\sum_{j=0}^{l(m)}\epsilon_j(m)\ \leq\ \frac{l(m)}{\beta}\hspace{0.2cm}\Leftrightarrow\hspace{0.2cm}
 	\sum_{j=0}^{l(m)}\epsilon_j(m)\ \leq t
 \end{equation}
 since the $\epsilon_{j}$'s are integers and $l(m)\in J(t)$ implies
 \[
 t\leq \ \frac{l(m)}{\beta}\  <\ t+1.
 \]

	Our strategy is  to bound the number of integers  $m'=\sum_{j=0}^{\lceil t\beta \rceil}\epsilon_j(m') b^{dj}$ in 
	\[
	A(t):=\Big\{ m'\in\N\ |\ m'<b^{d\lceil t\beta \rceil+d},
	\epsilon_j(m')\in\{0,1\},
	\epsilon_{\lceil t\beta \rceil}(m')=1,
	\sum_{j=0}^{\lceil t\beta \rceil}\epsilon_j(m')\leq t \Big\}.
	\] 
It follows that $A(t)\subset 	\mathcal{B}_b\big(\delta,b^{d\lceil t\beta \rceil+d}\big)$, since  $l(m')=\lceil t\beta \rceil$ by design and hence $l(m')\in J(t)$ such that, again by design, Equation \eqref{eq_Thm3boundbyt} holds and thus Equation \eqref{eq_proofBoundsBbdelta} is satisfied.
	 Hence our task is to determine the size of $A(t)$, which is equivalent to answering the combinatorial  question of
	how many ways there are to fill up to $(t-1)$ slots given $\lceil t\beta \rceil-1$ slots  to choose from. This number is given by 
\begin{equation}\label{eq_SumBinomsBound}
	|A(t)|=\sum_{j=0}^{t-1}\binom{\lceil t\beta \rceil-1}{j}>
	\binom{\lceil t\beta \rceil-1}{t-1}.
\end{equation}
	The rightmost binomial expression above will be bound with the following result.
	
	\vspace{0.2cm}
		
	\textbf{Lemma A} (Worsch \cite[Lem.2.5]{Worsch}) For every $\epsilon>0$ and $a\geq 2$ there is a $\nu_0(\epsilon)\in\N$ such that if $\nu>\nu_0(\epsilon)\cdot  a$, then
		\[
		 \frac{1}{1+\epsilon}\frac{1}{\sqrt{2\pi \nu}}\sqrt{\frac{a^2}{a-1}}C(a)^\nu<
		\binom{\nu}{\lfloor \frac{\nu}{a}\rfloor} <
		 \frac{1+\epsilon}{\sqrt{2\pi \nu}}\sqrt{\frac{a^2}{a-1}}C(a)^\nu,
		\]
		where $C(x)=g(x)\cdot g\big(\frac{x}{x-1}\big)$ with $g(x)=\sqrt[x]{x}$.
		
			\vspace{0.2cm}
		
		Let $N=\sum_{j=0}^n \epsilon_jb^{dj}$ with $n\in J(T+1)$ and $T-1\geq \nu_0(1)$ as in Lemma A.
		Since $\lceil T\beta\rceil +1\leq  (T+1)\beta \leq n$ 
		we  obtain with Lemma A ($\epsilon=1$, $a=\beta$ and $\nu=\lceil T\beta \rceil-1$) and
		\[
		\Big\lfloor\frac{\lceil T\beta \rceil-1}{\beta}\Big\rfloor=T-1,
		\]
		 that (by Equation \eqref{eq_SumBinomsBound})
		\[
		\big| \mathcal{B}_b(\delta,N)\big|\geq \big| \mathcal{B}_b(\delta,b^{d\lceil T\beta \rceil+d})\big|
		>\frac{1}{\sqrt{8\pi (\lceil T\beta \rceil-1)}}\sqrt{\frac{\beta^2}{\beta-1}}C(\beta)^{\lceil T\beta \rceil-1}.
		\]
	 For $n\in J(T+1)$, we have $(T+1)\beta\leq n\leq \log(N)/\log(b^d)< (T+2)\beta+1$ and the claim  of Theorem \ref{thm_SizeOfSetsB} follows after some elementary manipulations. \qed

\section{Proofs for the Applications}\label{sec_proofApplications}
 It is clear that the $L^p$--discrepancy is $f$--subadditive with $f(n)=\max\{n^{p-1},1\}$ from the remarks after Definition \ref{def_fSubLinear}.  The reference measure $\mu$ on $[0,1)$ is$\dx$, while on $\s$ it is the normalized surface measure.

\begin{proof}[Proof of Theorem \ref{thm_ApplvdC}]
	With $v\in[0,1)$ and  $\omega(b,m,v)=[0,1)\cap \big((\Z+v)\cdot b^{-m}\big)$, we have
	\[
	 \scd(\omega(b,m,v))\leq b^{-m}
	\]
	and hence $h_m^\infty=1$, if $\mathscr{D}(E_{Z_N})=N\scd(E_{Z_N})$. Furthermore, for $p>0$,
	\begin{align*}
		&\lpd(\omega(b,m,v))=\int_0^1\big| b^{-m}\sum\chi_{[0,t)}(x_j)-t\big|^p\dt\\
		&=\int_0^{vb^{-m}}t^p\dt
		+\sum_{j=0}^{b^m-2}\int_{(v+j)b^{-m}}^{(v+j+1)b^{-m}}\big| b^{-m}(j+1)-t\big|^p\dt+\int_{1-(1-v)b^{-m}}^{1}\big| 1-t\big|^p\dt\\
		&=\frac{(vb^{-m})^{p+1}}{p+1}
		+\sum_{j=0}^{b^m-2}b^{-m}b^{-mp}\int_{0}^{1}\big|1-v-t\big|^p\dt+  \frac{((1-v)b^{-m})^{p+1}}{p+1}\\
		&=\frac{v^{p+1}+(1-v)^{p+1}}{p+1}\Big(b^{-m(p+1)}+\frac{b^m-1}{b^m}b^{-mp}\Big)=\frac{v^{p+1}+(1-v)^{p+1}}{p+1}b^{-mp}.
	\end{align*}
Thus,
\[
\frac{1}{2^p}\frac{1}{p+1}b^{-mp}\leq \lpd(\omega(b,m,v))\leq \frac{1}{p+1}b^{-mp}
\]
and hence $h_m^\infty\cdot(p+1)=1$ in this case, if $\mathscr{D}(E_{Z_N})=N^p\lpd(E_{Z_N})$.
An application of Theorem \ref{thm_main} with these $h_m^\infty$ then yields the claim.
\end{proof} 

\begin{proof}[Proof of Theorem \ref{thm_ApplNiederreiter}]
	With  $\omega=\omega(b,m)=\{x_1,\ldots,x_{b^m}\}$ a perturbed lattice set of resolution $b^m$ and $T_v$ a valid translate for $\omega$ with $v\in[0,1)$,  we have
	\[
	\scd\big(T_v(\omega(b,m))\big)\leq 2b^{-m},
	\]
	and hence $h_m^\infty=2$ in this case, if $\mathscr{D}(E_{Z_N})=N\scd(E_{Z_N})$. Furthermore, with  
	\[
	F_\omega(t):=b^{-m}\sum_{j=1}^{b^m}\chi_{[0,t)}(x_j)-t,
	\]
	we have 
	\[
	\lim_{\tau\ra0^+}F_\omega(j/b^m-\tau)=0
	\]
	 for all $1\leq j\leq b^m$ and $|F_\omega(t)|\leq b^{-m}$ since $\omega(b,m)$ is a perturbed lattice set. Thus, with  $\omega'=b^{-m}\Z\cap[0,1)$, we see that for any $p>0$ and $0\leq j<b^m$ 
	 \begin{equation}\label{eq_proofTheorem5}
	 		\int_{j/b^m}^{(j+1)/b^m}|F_\omega(t)|^p\dt\leq \int_{j/b^m}^{(j+1)/b^m}|F_{\omega'}(t)|^p\dt,
	 \end{equation}
	simply because the graphs of $F_{\omega'}$ and $F_{\omega}$ are saw--tooth functions.
	What we really need is a bound on
	\[
	\int_0^1 |F_{T_v(\omega)}(t)|^p\dt.
	\]
	If $m=0$, then $(p+1)^{-1}$ is an upper bound as we see from Equation \eqref{eq_proofTheorem5}. For $m>0$ we introduce a  limit case of a perturbed lattice set: $\omega''=\{x_1'',\ldots,x_{r}''\}$ with $x_j''=x_{j+1}''=(j-1)/b^m$ if $j=1\mbox{ mod }2$ and $1\leq j\leq b^m$. This is a set of equally spaced pairs of points with distance $2/b^m$ between consecutive pairs, starting at the origin. After some thought, it should be clear that 
	\[
		\lpd\big(T_v(\omega(b,m))\big)=\int_{0}^{1}|F_{T_v(\omega)}(t)|^p\dt\leq \int_{0}^{1}|F_{\omega''}(t)|^p\dt=\lpd(\omega'').
	\]
	A similar computation as in the proof of Theorem \ref{thm_ApplvdC} yields
	\begin{align*}
		\lpd(\omega'')&\leq\sum_{j=0}^{\lceil b^m/2\rceil-1}\int_{2j/b^m}^{2(j+1)/b^m}\big| 2(j+1)b^{-m}-t\big|^p\dt\\
		&=\sum_{j=0}^{\lceil b^m/2\rceil-1}b^{-m}b^{-mp}\int_{0}^{2}\big| 2-t\big|^p\dt\leq b^{-mp}\frac{2^{p}}{p+1}.
	\end{align*}
	Hence, $h_m^\infty\cdot(p+1)\leq 2^p$ in this case, if $\mathscr{D}(E_{Z_N})=N^p\lpd(E_{Z_N})$.
	An application of Theorem \ref{thm_main} with these $h_m^\infty$ then yields the claim.
\end{proof}

 \textbf{Theorem B} (\cite[Thm.1]{Ferizovic} DF '22)
 	Let  $\omega_{b^m}$ be a perturbed lattice set of resolution $b^m$ and $T_v$ a valid translate for $\omega_{b^m}$, then, for $L:[0,1)^2\ra\s$ the Lambert projection, we have
 	$$\scd(L(T_v(\omega_{b^m})))\ \leq\ \frac{2+3\sqrt{2}}{b^m}+O(b^{-2m}).$$
 	\begin{proof}
 		Note that $T_v(\omega_{b^m})$ might not be a perturbed lattice
 		and hence \cite[Thm.1]{Ferizovic} doesn't apply straightforwardly, but we will describe the changes that are necessary to adapt it to our current needs. In the notation of \cite{Ferizovic}  we fix $Q=\fatone$, the identity matrix, and $N^\fatone(b^m)=\#|T_v(\omega_{b^m})|$. Then $b^{2m}-b^m\leq N^\fatone(b^m)\leq b^{2m}$ since up to $b^m$ many points of $T_v(\omega_{b^m})$ could lie on the $x$-axis  which is excluded from the domain $I^2$ there, such that $d^\fatone(b^m)\leq 1$ since $T_v(\omega_{b^m})\subset [0,1)^2$ by assumption.  
 		 The discrepancy estimates for caps in \cite{Ferizovic}  do not change if we consider the following tiling instead (note the shift by $v$, see the right hand side of Figure \ref{fig:ProofThmB}):
 		\[
 		[0,1)^2=\bigcup_{p\in( v+\Z^2)} b^{-m}\big(p+[0,1)^2\big)\cap[0,1)^2.
 		\]
 		Thus the discrepancy contribution from the  cap boundaries is still bounded by $C^\fatone \sqrt{2} b^{-m}\leq 3\sqrt{2} b^{-m}$, which follows from Lemma 2.2. of \cite{Ferizovic}. The last term of  \cite[Thm.1]{Ferizovic} that we need to take care of is $M^\fatone(b^m)$, which is the discrepancy contribution induced by the projection of the unit-square to the sphere. By its definition, it follows that $M^\fatone(b^m)\leq 1$ (see also Figure \ref{fig:ProofThmB}). This inequality holds for $\omega_{b^m}$ since at most $b^m$ points could lie on the $x$-axis of $I^2$ and the vertical boundary lines of $I^2$ do not give any contribution at all since $Q=\fatone$. This also holds for $T_v(\omega_{b^m})$ since the horizontal shift (the $x$-component of $v$) moves $\omega_{b^m}$ to the right, and the tiles $b^{-m}\big(p+[0,1)^2\big)\cap[0,1)^2$ that intersect the vertical line $\{1\}\times[0,1]$ each have a point mass from $T_v(\omega_{b^m})$, but the tiles  $b^{-m}\big(p+[0,1)^2\big)\cap[0,1)^2$ that intersect the vertical line $\{0\}\times[0,1]$ have no points at all. Hence the contribution from the vertical boundary lines of $I^2$ is zero.
 	\end{proof}
\begin{figure}[h!]
	\begin{center}
		\includegraphics[scale=0.3]{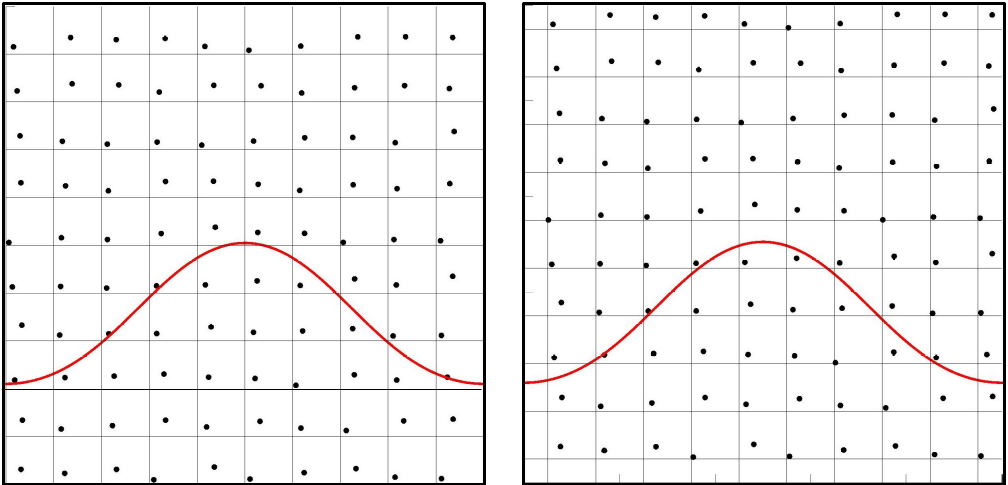}
		\caption{This visual aid intends to help the reader verify that $M^\fatone(b^m)\leq 1$ in the proof of Theorem B. The curve represents $L^{-1}(\partial C)$ of a spherical cap $C$. The dots represent $\omega_{b^m}$ and $T_v(\omega_{b^m})$ respectively.}
		\label{fig:ProofThmB}
	\end{center}
\end{figure}

\begin{proof}[Proof of Theorem \ref{thm_AppS2}]
 Theorem B    yields bounds on the values $h_m^\infty$ with  $N=b^{2m}$:
	\[ 
	\mathscr{D}(E_{\omega_{N}})=b^{2m}\scd(L(\omega_{b^{2m}}))\ \leq\ b^{m}\big(2+3\sqrt{2}\big)+C.
	\]
	The existence of $C=C(\det(Q))$ follows from the proof of \cite[Thm. 1]{Ferizovic}. 
	 For general $N=\sum_{j=0}^{n}\epsilon_j\cdot b^{2j}$, we then apply Theorem \ref{thm_main},  Equations \eqref{eq_mainResultN} and \eqref{eq_mainResultNprime} to obtain the first inequality.
	 The fact $\sqrt{N}\geq b^{n}$ proves the second inequality.
\end{proof}

\begin{figure}[h!]
	\begin{center}
		\includegraphics[scale=0.25]{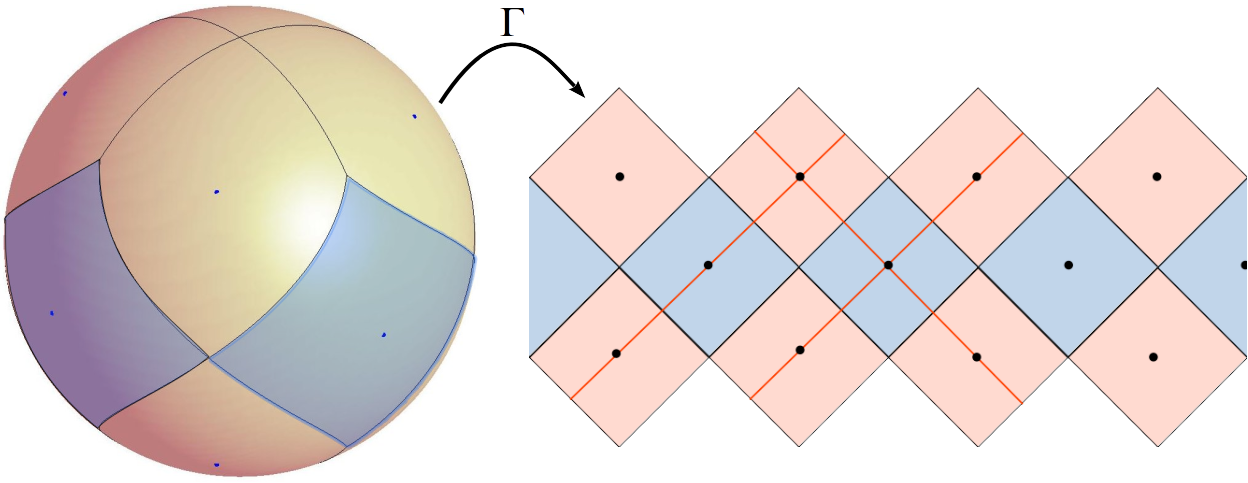}
		\caption{ A visual representation of the HEALPix projection map $\Gamma:\s\ra R$.}
		\label{fig:HealpixGamma}
	\end{center}
\end{figure}


\textbf{Theorem C} (\cite[Thm.1]{FerHofMas} DF, Hofstadler, Mastrianni '22)
Let $v\in\R^2$. Consider the HEALPix projection map $\Gamma:\s\ra R\subset \C\equiv \R^2$ for $N=12\cdot 4^m$ and define 
\[
H^v_N =\Gamma^{-1}\Big(\Big[\frac{\sqrt{2}}{4}\frac{1}{2^m}e^{i\pi/4}\big(\Z+i\Z\big)+v\Big]\cap R\Big),
\]
where we abuse notation a bit when we identify $\R^2$ with $\C$.
$H^0_N$ is then  the set of centers of pixels in a HEALPix tessellation of $\s$ with $N$ pixels. Then 
\[
\frac{1}{\sqrt{3N}}\ \le\ \scd(H^v_N) \ \le\  \frac{2}{\sqrt{3}}\frac{5+\sqrt{2}}{\sqrt{N}}  + \frac{1000}{N}.
\]
\begin{proof}
	The proof is verbatim the same as in \cite{FerHofMas}, since the discrepancy is treated by separate investigation of area and boundary contribution from caps. The position of points plays no role whatsoever, as long as one pixel is occupied by exactly on point, and hence is valid for translates too.
\end{proof}

\begin{proof}[Proof of Theorem \ref{thm_AppHealpix}]  We need the HEALPix projection $\Gamma^{-1}$ (see Section 3 of \cite{FerHofMas} and Figure \ref{fig:HealpixGamma}) with domain consisting of 12 squares of side-length $\sqrt{2}/2$. We first consider a set-valued map $\gamma:[0,1)^2\ra\mbox{dom}(\Gamma^{-1})$, where $\gamma(x)$ is the set of 12 copies (in the obvious way) of $x$ scaled by $\sqrt{2}/2$ and rotated by 45 degrees. Then we define a set-valued $G:[0,1)^2\ra\s$ by
	\begin{equation}\label{eq_defGprojection}
		G(x)=\Gamma^{-1}\circ \gamma(x).
	\end{equation}
	For a sequence $S=(s_0,\ldots)=S^\Box_{2,2}\subset[0,1)^2$ and $n\in\N$ let
	\[
	G(s_{n-1})=\{q^n_1,\ldots,q^n_{12}\}\subset\s.
	\]
	We then write $Z=G(S)$ for the sequence
	\[
	Z=(z_1,z_2,\ldots)=(q^1_1,\ldots,q^1_{12},q^2_1,\ldots,q^2_{12},q^3_1,\ldots,q^3_{12},\ldots)
	\] 
	 where the elements $\{q^n_1,\ldots,q^n_{12}\}$ are ordered arbitrarily.
	Then for $K\in\N$
	\[
	G(S_{K})=Z_{12 K}=(q^1_1,\ldots,q^1_{12},\ldots,q^{K}_1,\ldots,q^{K}_{12})=(z_1,\ldots,z_{12 K}).
	\]
		Define
		\[
		\mathscr{D}(E_{S_{K}})=12 K\cdot\scd (G(S_K)).
		\]
		This is $1$--subadditive over the empirical measures, since for finite\footnote{We treat $A\cup B$ as a multiset.}  $A,B\subset[0,1)^2$
		\begin{align*}
			\mathscr{D}(E_{A\cup B})&=12\big(|A|+|B|\big)\cdot\scd (G(A\cup B))\\
			&=\sup_{F\in T}\Bigg|\sum_{j=1}^{12(|A|+|B|)}\chi_{F}(x_j)-12\big(|A|+|B|\big)\mu(F)\Bigg|\\
			&\leq 12|A|\cdot\scd (G(A))+12|B|\cdot\scd (G( B)).
		\end{align*}
	Let $L_m=\ell^2_{2^m}$, then  Theorem C states that, for $N=12\cdot 4^m$,
		\[
		\sup_{(L_m,v)\in \Omega_m}\mathscr{D}\big(E_{T_{v}(L_m)}\big)= h_m^\infty\ \leq\ 2^m\big(5+\sqrt{2}\big)4+ 1000.
		\]
    Let $N=12 K+k$ be as in Theorem \ref{thm_AppHealpix}. Then with $R_k=\{z_{12K+1},\ldots,z_{12K+k}\}$
    \[
    \mathscr{D}\big(E_{Z_N}\big)\ \leq \ \mathscr{D}\big(E_{Z_{12K}}\big)+\mathscr{D}\big(E_{R_k}\big)
    \]
    and the last term is bounded by a universal constant independent of $K$ and $k$. The claim then follows from Theorem \ref{thm_main} and 
	the trivial estimate $\sqrt{12}\cdot 2^{\kappa}\leq\sqrt{N}$.
\end{proof}

\section{Extensions to Self-Similar Sets}\label{sec_MiscExten}

The underlying idea for our construction of the sequences $S^\Box$ and $S^\boxplus$ can be extended to more general sets than hyper-cubes. Indeed, after the first version of this text was uploaded to arXiv, we found the work by Basu \& Owen \cite{BasuOwen} who defined a sequence of points in a triangle. Later still we found the work of Infusino and  Volčič \cite{InfusinoVolcic} who gave the same method based on Hutchinson's results as we do here (that's why we recognized our construction in theirs in the first place), but our algorithm is slightly more general.

\vspace{0.1cm}

The next definition is a combination of notions as introduced in Hutchinson \cite{Hutchinson}, where we added the restriction that $\mathscr{H}_n(K)<\infty$ which follows from the results anyway.

\begin{definition} Let  $\mathscr{I}=\{A_1,\ldots, A_k\}$  be a set of \textbf{contraction maps} on a complete metric space $(X,\delta)$, i.e. $A_j:X\ra X$ is Lipschitz continuous with Lipschitz constant $\lambda_j<1$. 
	$\mathscr{I}$ is said to satisfy the \textbf{open set condition} if there is a non-empty, open set $U\subset X$ with 
	\[
	\bigcup_{j=1}^k A_j(U)\subset U, \hspace{0.3cm}\mbox{and}\hspace{0.3cm}A_j(U)\cap A_s(U)=\emptyset,\hspace{0.1cm}\mbox{if }j\neq s.
	\]	
	A non-empty set $K\subset X$ is called \textbf{invariant} with respect to  $\mathscr{I}$ if 
	\[
	K=\bigcup_{j=1}^k A_j(K).
	\]
	We call $K$ \textbf{self-similar} wrt. $\mathscr{I}$ if $0<\mathscr{H}_n(K)<\infty$ for some $n>0$, where $\mathscr{H}_n$ denotes the $n$-dimensional Hausdorff measure, and for $j\neq s$ we have
	\[
	\mathscr{H}_n\big[A_j(K)\cap A_s(K)\big]=0.
	\]
	
\end{definition}

\textbf{Theorem D} (Hutchinson \cite[p.713]{Hutchinson}) Let  $\mathscr{I}=\{A_1,\ldots, A_k\}$   be a set of contraction maps on a complete metric space $(X,\delta)$, then there exists a \textit{unique, closed, bounded and compact set} $K_{\mathscr{I}}\subset X$ which is invariant with respect to  $\mathscr{I}$.
\vspace{0.1cm}

\textbf{Theorem E} (Hutchinson \cite[p.714]{Hutchinson}) Let $(X,\delta)=(\R^d,\|\cdot\|)$ and $\mathscr{I}=\{A_1,\ldots, A_k\}$   be a set of maps satisfying $A_j=\sigma_j\circ \tau_j\circ O_j$ where $\sigma_j(x)=r_j\cdot x$ for $r_j\in[0,1)$, $\tau_j(x)=x-b_j$ for some $b_j\in\R^d$ and $O_j$ is an orthonormal map (see  \cite[p.717]{Hutchinson}). Let $\mathscr{I}$ satisfy the open set condition. Then from Theorem D we obtain the set $K_{\mathscr{I}}$ which  is now guaranteed to be self-similar with Hausdorff dimension given by the unique number $n$, such that
\[
\sum_{j=1}^k r_j^n=1.
\]

\begin{definition}\label{def_selfSimilarSet}  Let  $\mathscr{I}_n=\{A_1,\ldots, A_k\}$ be as in Theorem E with $r_j=k^{-1/n}$ and let $\Omega\subset \R^d$ denote the unique, compact, self-similar set, invariant with respect to $\mathscr{I}_n$ and necessarily of Hausdorff dimension $n$.
\end{definition}
\begin{remark}
	It follows easily that $\mathscr{H}_n\big(A_j(\Omega)\big)=\tfrac{1}{k}\mathscr{H}_n(\Omega).$ Rectangles and triangles embedded in $\R^d$ are such $\Omega$, so are $k$-reptiles (see  Matoušek \& Safernová \cite{Matousek}). 
\end{remark}

 Figure \ref{fig:TriangleExample} illustrates two examples for $\Omega$ being the triangle $\triangle$ with corners at $v_1=(0,0)^t, v_2=(1,0)^t, v_3=(0,1)^t$, where $A_j(x)=\onehalf M_j(x)+\onehalf v_j$ with  $M_t=\fatone$  (= identity matrix) for $t<4$ and $M_4=-\fatone$, $ v_4=(1,1)^t$. (For every other triangle $T$ we have $T=M\triangle+v$ for some invertible matrix $M$ and vector $v$.)

\begin{algorithm*}
	Let $\Omega$   and $A_1,\ldots, A_k$ be as in Definition \ref{def_selfSimilarSet}.  We  define a sequence $S=(s_0,s_1,\ldots)$ in $\Omega$ recursively as follows: Choose $s_0\in\Omega$ arbitrarily and fix a sequence $(\pi_1,\pi_2,\ldots)$ of permutations of $k$ elements. Define for $1\leq m\leq k$
	\[
	V_m=A_{\pi_1(m)}(\Omega).
	\]
	Choose $q_m\in V_m$ arbitrarily. Choose the smallest $p$ such that $s_0\in V_p$ for some $p\in\{1,\ldots,k\}$. Define  $s_j=q_j$ for $j<p$ and $s_{j-1}=q_{j}$ for $j>p$ and $0<j\leq k$.
	
	Suppose  elements $\{s_0,\ldots,s_{k^t-1}\}$ and   subsets $V_1,\ldots,V_{k^t}$ have been defined. Then, for $1\leq m\leq k^t$ and $1\leq r\leq k$ choose arbitrary $q^r_{m}$  such that
	\[
	q^r_{m}\in V_{(r-1)\cdot k^t+m}:=A_{\pi_{t+1}(r)}(V_m).
	\]
	There are $1\leq p_1<\ldots <p_{k^t}\leq k^{t+1}$ with $s_j\in V_{p_{j+1}}$. Consider the  vector:
	\[
	\big(q^1_1,q^2_1,\ldots,q^k_1,q^1_2,q^2_2,\ldots,q^k_2,q^1_3,q^2_3,\ldots,q^k_3,\ldots,q^1_{k^t},q^2_{k^t},\ldots,q^k_{k^t}\big)
	\]
	from which we remove elements $q^{r}_{u}$ with $ [(r-1)\cdot k^t+u]\in\{p_1,\ldots,p_{k^t}\}$ to obtain   
	\[
	(s_{k^t},\ldots, s_{k^{t+1}-1});
	\]
	i.e. the next elements of $S$. (Thus $s_{k^t}=q^1_1$, unless $1=p_j$ for some $j$, etc.) \null\hfill$\Diamond$
\end{algorithm*}

\begin{figure}[h!]
	\begin{center}
		\includegraphics[scale=0.36]{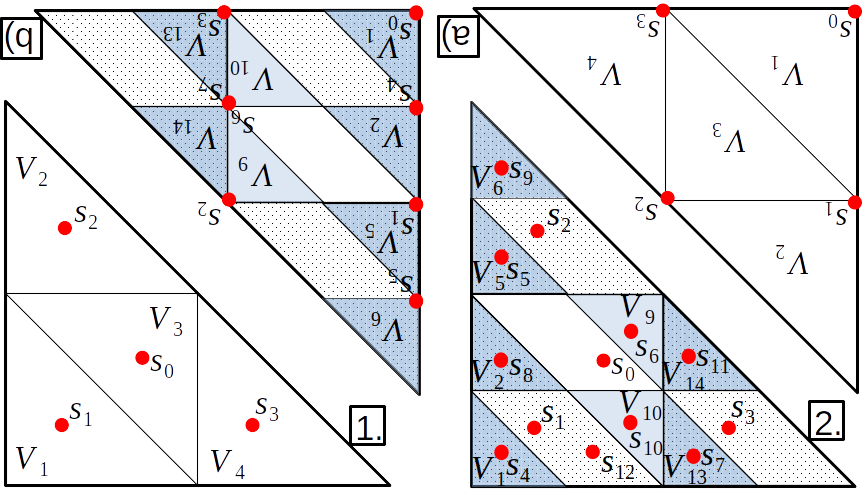}
		\caption{ Triangles 1./ 2. show the sequence when $q^r_m$ are centroids. In Triangles \textrm{a})/\textrm{b})  the $q^r_m$ are corners-- where we see that $s_6=s_7$.}
		\label{fig:TriangleExample}
	\end{center}
\end{figure}

 \begin{acknowledgement}
	I thank Ryan Matzke for helping me improve the language of the text and for suggesting to state Theorems B and C explicitly which made the proofs of Theorem \ref{thm_AppS2} and \ref{thm_AppHealpix} easier to read. I thank the anonymous referees for their time and helpful comments.
\end{acknowledgement}

%



\begin{thebibliography}{10} 
	
	\bibitem{Aistleitner} C. Aistleitner, J.S. Brauchart, J. Dick: \emph{Point Sets on the Sphere $S_2$ with Small
		Spherical Cap Discrepancy}. Discrete Comput Geom 48, pp. 990-1024 (2012). DOI: 10.1007/s00454-012-9451-3
	
	\bibitem{AlishahiZamani} K. Alishahi, M. Zamani: \emph{The spherical ensemble and		uniform distribution of points on the sphere}.
	Electron. J. Probab. 20 (23), pp. 1-22 (2015). 
	
	\bibitem{ArnoldKrylov} V.I. Arnol'd, A.L. Krylov: \emph{Uniform distribution of points on a sphere and some ergodic properties of Solutions of linear ordinary differential equations in a complex region}. [Translation into English published in Soviet. Math. Dokl. 4 (1963)] Dokl. Akad.
	Nauk SSSR 148:1, pp. 9-12 (1963).
	


	
	
	\bibitem{BasuOwen} K. Basu, A.B. Owen: \emph{Low Discrepancy Constructions in the Triangle}. SIAM J. Numer. Anal. 53(2), pp. 743-761 (2015). DOI: 10.1137/140960463
	
	
	\bibitem{Bauer} R. Bauer: \emph{Distribution of points on a sphere with application to star catalogs}. J. Guid. Cont. Dyn.  23(1), pp. 130-137 (2000). DOI: 10.2514/2.4497
	
	
	\bibitem{Beck} J. Beck: \emph{Sums of distances between points on a sphere—an application of the theory of irregularities of	distribution to discrete geometry}. Mathematica 31(1), pp. 33–41 (1984). DOI: 10.1112/S0025579300010639
	
	
	

		\bibitem{BejianFaure}R. Béjian, H. Faure: \emph{Discrépance de la suite de van der Corput}. Séminaire Delange-Pisot-Poitou (Théorie des
	nombres) 13, pp.1–14  (1977/78).
	
	
	\bibitem{BeltranEtayo} C. Beltrán, U. Etayo: \emph{The Diamond ensemble: A constructive set of spherical points with small logarithmic energy}. J. Complex. 59, 101471 (2020).
	
	\bibitem{BeltranEtayoLopez} C. Beltrán, U. Etayo, P. L\'{o}pez-G\'{o}mez: \emph{Low energy points on the sphere and on the real projective plane}.
	Journal of Complexity 76, 	101742 (2023).
	https://doi.org/10.1016/j.jco.2023.101742
	
	\bibitem{BeltranFerizovicLopez} C. Beltrán, D. Ferizovi\'{c}, P. L\'{o}pez-G\'{o}mez: 
	\emph{Measure-preserving mappings from the unit cube to some symmetric spaces}. Journal of Approximation Theory 308 (2025).  https://doi.org/10.1016/j.jat.2025.106145
	
	\bibitem{BeltranMarzoOrtega} C. Beltrán, J. Marzo, J. Ortega-Cerdà: \emph{Energy and discrepancy of rotationally invariant determinantal point processes in high dimensional spheres}. J. Complex. 37, pp. 76-109 (2016). DOI: 10.1016/j.jco.2016.08.001
	
	\bibitem{Beretti} T. Beretti: \emph{On the distribution of the van der Corput sequences}. Arch. Math. 120, 283–295 (2023). DOI: 10.1007/s00013-022-01811-4
	


	\bibitem{BrauchartGrabner} J. Brauchart, P. Grabner: \emph{Distributing many points on spheres: Minimal energy and designs}. J. Complex., 31(3), pp. 293-326 (2015). 
	
	\bibitem{BrauchartSaffSloanWomersly} J. Brauchart, E.B. Saff, I.H. Sloan, R.S. Womersley: \emph{QMC Designs: Optimal order Quasi Monte Carlo integration schemes on the sphere}. Math. Comp. 83(290), pp. 2821--2851 (2014). DOI: S 0025-5718(2014)02839-1
	
	\bibitem{BrownSteinerberger} L. Brown, S. Steinerberger: \emph{On the Wasserstein distance between classical sequences and the Lebesgue measure}.	Trans. Amer. Math. Soc. 373, pp.  8943-8962  (2020). DOI: 10.1090/tran/8212
	

	
	
	
	\bibitem{CuiFreeden} J. Cui, W. Freeden: \emph{Equidistribution on the sphere}. SIAM J. Sci. Comput. 18(2), pp. 595-609 (1997). DOI: 10.1137/S1064827595281344
	

	\bibitem{Etayo}U. Etayo: \emph{ Spherical Cap Discrepancy of the Diamond Ensemble}.  Discrete Comput Geom 66, pp. 1218–1238 (2021). DOI: 10.1007/s00454-021-00305-4
	
\bibitem{Faure}	H. Faure: \emph{Discrepancy and diaphony of digital (0, 1)-sequences in prime base}. Acta Arith. 117, pp. 125--148 (2005). DOI: 10.4064/aa117-2-2 


\bibitem{FaureKritzerPillichshammer} H. Faure, P. Kritzer, F. Pillichshammer: \emph{From van der Corput to modern constructions of sequences for quasi-Monte Carlo rules}. Indagationes Mathematicae 26(5), pp. 760-822 (2015). DOI: 10.1016/j.indag.2015.09.001
	
	\bibitem{Ferizovic} D. Ferizovi\'{c}: \emph{Spherical cap discrepancy of perturbed lattices under the Lambert projection}. Discrete Comput Geom 71, pp. 1352–1368 (2024). https://doi.org/10.1007/s00454-023-00547-4
	
	\bibitem{FerHofMas} D. Ferizovi\'{c}, J. Hofstadler, M. Mastrianni: \emph{The spherical cap discrepancy of HEALPix points}. Studia Sci. Math. Hungarica  (CoGeTo) 60(4), pp. 249-273 (2024). 
	https://doi.org/10.1556/012.2023.04299 
	
	\bibitem{FerMatzke} D. Ferizovi\'{c}, R. Matzke: \emph{Bounds on the Wasserstein $p$-metric of generalized van der Corput sequences}. (work in progress)
	
	\bibitem{Gorski} K.M. Górski, E. Hivon, A.J. Banday, B.D. Wandelt, F.K. Hansen, M. Reinecke, M. Bartelmann: \emph{HEALPix: A Framework for High-Resolution Discretization and Fast Analysis of Data Distributed on the Sphere}. The Astrophysical Journal 622, pp. 759-771 (2005). DOI: 10.1086/427976
	
	
	\bibitem{GriepentrogEtAl} J.A. Griepentrog, W. H\"{o}ppner, H.-C. Kaiser, J. Rehberg: \emph{A bi-Lipschitz continuous, volume		preserving map from the unit ball		onto a cube}. Note di Matematica 28(1), pp. 185-201 (2008). DOI:  10.1285/i15900932v28n1p177 
	

	
	
	\bibitem{Hardin} D.P. Hardin, T. Michaels, E.B. Saff: \emph{A Comparison of Popular Point Configurations on $\s$}. Padova University Press, Dolomites Research Notes on Approximation 9, pp. 16-49 (2016). 	DOI: 10.14658/pupj-drna-2016-1-2

	\bibitem{HlawkaBinder} E. Hlawka, C. Binder: \emph{\"{U}ber die Entwicklung der Theorie der gleichverteilung in den Jahren 1909 bis 1916}. Arch. Hist. Exact Sci. 36, pp. 197–249 (1986). 
	
	
    \bibitem{Holhos} A.	Holho{\c s}, D. Ro{\c s}ca:\emph{ Uniform refinable 3D grids of regular convex polyhedrons and balls}. Acta Math. Hungar. 156, pp. 182–193 (2018). 
	
	\bibitem{Hutchinson} J.F. Hutchinson: \emph{Fractals and Self Similarity}. Indiana University Mathematics Journal 30(5), pp. 713--774  (1981). DOI: 10.1512/iumj.1981.30.30055
	
	
	
	
	\bibitem{InfusinoVolcic} M. Infusino, A. Volčič: \emph{Uniform distribution on fractals}. Uniform Distribution Theory 4 (2), pp. 47–58 (2009).
	

		
	\bibitem{KuijlaarsSaff} A.B.J. Kuijlaars, E.B. Saff: \emph{Distributing many points on a sphere}. The Mathematical Intelligencer 19, 5–11 (1997). DOI: 10.1007/BF03024331
	
		\bibitem{LeobacherPillichshammer} G. Leobacher, F. Pillichshammer: \emph{Bounds for the weighted Lp discrepancy and tractability of integration}. Journal of Complexity 19(4), 529-547 (2003). DOI: 10.1016/S0885-064X(03)00009-8
	
	
	\bibitem{LopGarciaMcCleary} A. López-García, R.E. McCleary: \emph{	Asymptotics of greedy energy sequences on the unit circle and the sphere}.
	J. Math. Anal. Appl.  504(1), 125269 (2021)
	
	
	\bibitem{LubotzykPhillipsSarnak} A. Lubotzky, R. Phillips, P. Sarnak: \emph{Hecke Operators and Distributing Points on the Sphere I}. Commun Pure Appl Math  39 (1968). 
	
		\bibitem{Matousek} J. Matoušek, Z. Safernová: \emph{On the Nonexistence of k-reptile Tetrahedra}. Discrete Comput Geom 46, pp. 599–609 (2011).
	
	\bibitem{Niederreiter} H. Niederreiter: \emph{Random Number Generation and Quasi-Monte Carlo Methods}
	SIAM, Philadelphia (1992).
	
	\bibitem{Pillichshammer} F. Pillichshammer: \emph{On the discrepancy of (0,1)-sequences}. J. Number. Theory 104, pp. 301–314 (2004). DOI: 10.1016/j.jnt.2003.08.002
	
	
	\bibitem{PollicottSewell} M. Pollicott, B. Sewell: \emph{An infinite interval version of the $\alpha$-Kakutani
	equidistribution problem}. Isr. J. Math. 260, 365–399 (2024). https://doi.org/10.1007/s11856-023-2569-6
	
	

	\bibitem{Skriganov} M.M. Skriganov: \emph{Bounds for Lp-Discrepancies of Point Distributions in Compact Metric Measure Spaces}. Constr Approx 51, 413–425 (2020). https://doi.org/10.1007/s00365-019-09476-z
	
	
	\bibitem{SloanWozniakowski} I.H. Sloan, H. Woźniakowski: \emph{When Are Quasi-Monte Carlo Algorithms Efficient for High Dimensional Integrals?}. Journal of Complexity 14(1), pp. 1-33 (1998). https://doi.org/10.1006/jcom.1997.0463
	
	
	\bibitem{Spiegelhofer} L. Spiegelhofer: \emph{Discrepancy results for the van der Corput sequence}. Unif. Distrib. Theory 13(2), pp. 57-69 (2018). DOI: 10.2478/udt-2018–0010
	
	
	\bibitem{Stolarsky} K.B. Stolarsky: \emph{Sums of distances between points on a sphere II}. Proc. of the AMS 41(2), pp. 575-582 (1973). 
	
	\bibitem{TijdemanWagner} R. Tijdeman, G. Wagner: \emph{A sequence has almost nowhere small discrepancy}.
	Monatsh. Math. 90(4), pp. 315-329 (1980). DOI:10.1007/BF01540851
	
	\bibitem{vanderCorput} J.G. van der Corput: \emph{Verteilungsfunktionen I–II}, Proc. Akad. Amsterdam 38, pp. 813–821 \& pp. 1058–1066  (1935).
			
	\bibitem{Wang} H.H.	Wang: \emph{On vectorizing the fast fourier transform}. BIT 20, pp. 233–243 (1980). DOI: 10.1007/BF01933196
			
	\bibitem{Weyl} H. Weyl: \emph{\"{U}ber die Gleichverteilung von Zahlen mod. Eins}. Math. Ann. 77, pp. 313-352 (1916). DOI: 10.1007/BF01475864
	
	\bibitem{Worsch} T. Worsch: \emph{Lower and upper bounds for (sums of) binomial coefficients}. (1994). DOI:  10.5445/IR/181894
	
	%
	
	
	
	
	
	
	
	




%
\end{thebibliography}
\end{document}